\documentclass[a4paper,12pt]{amsart}

\usepackage[latin1]{inputenc}
\usepackage{graphicx}
\usepackage{amssymb}
\usepackage{amsopn}
\usepackage[line,matrix,frame,curve,poly]{xy}
\usepackage{color}
\usepackage{multirow}
\usepackage{url}

\newtheorem{theorem}{Theorem}[section]
\newtheorem{conjecture}[theorem]{Conjecture}
\newtheorem{proposition}[theorem]{Proposition}
\newtheorem{corollary}[theorem]{Corollary}
\newtheorem{lemma}[theorem]{Lemma}

\theoremstyle{definition}
\newtheorem{definition}[theorem]{Definition}

\newtheorem*{remark}{Remark}	% wenn man dem Leser etwas mitteilen moechte
\newtheorem*{note}{Note}			% wenn man auf etwas hinweisen moechte

\newcommand{\Dfn}[1]{{\sf #1}} % for definitions

\def\UV{\operatorname{U}(V)}
\def\Sn{\mathcal{S}_n}
\def\End{\operatorname{End}}
\def\coh{\operatorname{coh}}
\def\Cat{\operatorname{Cat}}
\def\area{\operatorname{area}}
\def\gr{\operatorname{gr}}
\def\adh{\mathtt{ad}\hspace*{1pt} \mathbf{h}}
\def\sgn{\operatorname{sgn}}
\def\inv{\operatorname{inv}}
\def\alt{\operatorname{alt}}
\def\Shi{\mathsf{Shi}}
\def\Dn{\mathcal{D}_n}
\def\An{A_{n-1}}
\def\Hc{{\sf H}_c}

\newcommand{\Hm}[1]{{\sf H}^{(#1)}}
\newcommand{\Lm}[1]{L^{(#1)}}
\newcommand{\Mm}[1]{\mathcal{M}^{(#1)}}

\def\DR{D\!R}

\def\h{V}
\def\x{\mathbf{x}}
\def\y{\mathbf{y}}
\def\m{\mathbf{m}}
\def\C{\mathbb{C}}
\def\R{\mathbb{R}}
\def\Z{\mathbb{Z}}
\def\N{\mathbb{N}}
\def\qstat{\operatorname{qstat}}
\def\tstat{\operatorname{tstat}}
\def\e{\mathbf{e}}
\def\det{\epsilon}
\newcommand{\mdet}[1]{\epsilon^{\otimes #1}}

\def\totdeg{\deg_{t}}
\def\hdeg{\deg_{\mathbf{h}}}
\def\Fdot{\operatorname{F}_\bullet}

\def\determinant{\operatorname{det}}
\def\eeps{\mathbf{e_{\det}}}

\def\span{\operatorname{span}}
\def\mod{\operatorname{mod}}

\definecolor{sehrhellgrau}{rgb}{.80,.80,.80}
\definecolor{hellgrau}{rgb}{.6 , .6 , .6}
\definecolor{grau}{rgb}{.5 , .5 , .5}
\definecolor{dunkelgrau}{rgb}{.35 , .35 , .35}
\definecolor{schwarz}{rgb}{0 , 0 , 0}

\begin{document}

	\title{$q,t$-Fu\ss-Catalan numbers for finite reflection groups}
	\author{Christian Stump}
	\thanks{Research supported by the Austrian Science Foundation FWF, grants P17563-N13 and S9600.}

	\address{Fakult\"at f\"ur Mathematik, Universit\"at Wien, Nordbergstra{\ss}e 15, A-1090 Vienna, Austria}
	\email{christian.stump@univie.ac.at}
	\urladdr{http://homepage.univie.ac.at/christian.stump/}

	\date{\today}

	\begin{abstract}
		In type $A$, the $q,t$-Fu\ss-Catalan numbers can be defined as a bigraded Hilbert series of a module associated to the symmetric group. We generalize this construction to (finite) complex reflection groups and, based on computer experiments, we exhibit several conjectured algebraic and combinatorial properties of these polynomials with non-negative integer coefficients. We prove the conjectures for the dihedral groups and for the cyclic groups. Finally, we present several ideas how the $q,t$-Fu\ss-Catalan numbers could be related to some graded Hilbert series of modules arising in the context of rational Cherednik algebras and thereby generalize known connections.
	\end{abstract}

	\maketitle

	\section*{Acknowledgements}
		The author would like to thank Christian Krattenthaler and Vic Reiner as well as two anonymous referees for many useful comments and suggestions.

		\medskip

		The results in this paper are part of a PhD thesis written at the University of Vienna, Austria \cite{stump2} and were partly presented at the FPSAC 2008 conference in Vi{\~n}a del Mar, Chile \cite{stump3}.

	\section{Introduction}
		The $q,t$-Catalan numbers and later the $q,t$-Fu\ss-Catalan numbers arose within the last $15$ years in more and more contexts in different areas of mathematics, namely in \emph{symmetric functions theory}, \emph{algebraic} and \emph{enumerative combinatorics}, \emph{representation theory} and \emph{algebraic geometry}. They first appeared in a paper by M.~Haiman as the Hilbert series of the alternating component of the \emph{space of diagonal coinvariants} \cite{haiman2}. In \cite{garsiahaiman}, A.~Garsia and M.~Haiman defined them as a rational function in the context of \emph{modified Macdonald polynomials}. Later, in his work on the $n!$\emph{--} and the $(n+1)^{n-1}$\emph{--conjectures}, M.~Haiman showed that both definitions coincide \cite{haiman3}. J.~Haglund \cite{haglund2} found a very interesting combinatorial interpretation of the $q,t$-Catalan numbers which he proved together with A.~Garsia in \cite{garsiahaglund}. In \cite{loehr}, N.~Loehr conjectured a generalization of this combinatorial interpretation for the $q,t$-Fu\ss-Catalan numbers. This conjecture is still open.
		
		The $q,t$-Fu\ss-Catalan numbers have many interesting algebraic and combinatorial properties. To mention some: they are symmetric functions in $q$ and $t$ with non-negative integer coefficients and specialize for $q=t=1$ to the well-known \Dfn{Fu\ss-Catalan numbers}
		$$\Cat_n^{(m)} := \frac{1}{mn+1}\binom{(m+1)n}{n}.$$
		Furthermore, specializing $t=1$ reduces them to the combinatorial $q$-Fu\ss-Catalan numbers introduced by J.~F{\"u}rlinger and J.~Hofbauer in \cite{fuerlingerhofbauer}; specializing $t=q^{-1}$ reduces them, up to a power of $q$, to the $q$-Fu\ss-Catalan numbers introduced for $m = 1$ by P.A.~MacMahon in \cite[p.~1345]{macmahon2}.
		
		The Fu\ss-Catalan numbers $\Cat_n^{(m)}$ have a generalization to all well-generated complex reflection groups. A standard reference for background on real reflection groups is \cite{humphreys}, for further information on complex reflection groups see \cite{brouemallerouquier,chevalley,shephard,shephardtodd,springer}. Let $W$ be such a well-generated complex reflection group, having rank $\ell$, degrees $d_1 \leq \dots \leq d_\ell$ and Coxeter number $h := d_\ell$. The \Dfn{Fu\ss-Catalan numbers} associated to $W$ are then defined by
		$$\Cat^{(m)}(W) := \prod_{i=1}^\ell{\frac{d_i+mh}{d_i}}.$$
		In the case of $W = \An$, we have $\ell = n-1$, $d_i = i + 1$ and $h = n$. This gives
		$$\Cat^{(m)}(\An) = \Cat^{(m)}_n.$$

		For $m=1$, $\Cat^{(m)}(W)$ first appeared in \cite{reiner} where V.~Reiner proved for the classical reflection groups case-by-case that the number of \emph{non-crossing partitions} equals the number of \emph{non-nesting partitions}, and that both are counted by this product. In full generality of well-generated complex reflection groups, $\Cat^{(m)}(W)$ was considered by D.~Bessis in \cite{bessis4} where he studied chains in the non-crossing partition lattice.

		It turns out that the interpretation of the $q,t$-Fu\ss-Catalan numbers in terms of the space of diagonal coinvariants is attached to the reflection group of type $A$ whereas the other interpretations can --~so far~-- not be generalized to other reflection groups. We define the space of diagonal coinvariants for any (finite) complex reflection group and define $q,t$-Fu\ss-Catalan numbers in terms of this module. Moreover, we explore several conjectured properties of those polynomials in this generalized context. In particular, we conjecture that the $q,t$-Fu\ss-Catalan numbers reduce for well-generated complex reflection groups and the specialization $q = t = 1$ to the Fu\ss-Catalan numbers $\Cat^{(m)}(W)$.
		
		For real reflection groups, we finally explore connections between the $q,t$-Fu\ss-Catalan numbers and a module which naturally arises in the context of \emph{rational Cherednik algebras}. We construct a surjection from the space of diagonal coinvariants to the module in question. This construction was, for $m = 1$, exhibited by I.~Gordon in \cite{gordon}.
	
		For background on representation theory we refer to \cite{fultonharris}; for background on the classical $q,t$-Fu\ss-Catalan numbers, we refer to a series of papers and survey articles by A.~Garsia and M.~Haiman \cite{garsiahaiman,haiman2,haiman4,haiman5} and to a recent book by J.~Haglund \cite{haglund}.

		\bigskip

		This paper is organized as follows:

		In Section~\ref{qtfusscatA}, we recall some background on classical $q,t$-Fu\ss-Catalan numbers.

		In Section~\ref{qtfusscat}, we define $q,t$-Fu\ss-Catalan numbers for all complex reflection groups (Definition~\ref{def:qtfusscatreal}) and present several conjectures concerning them (Conjectures~\ref{conjecturedim}, \ref{conjectureqtdiag} and \ref{conjectureqtq}). Moreover, we explicitly compute the $q,t$-Fu\ss-Catalan numbers for the dihedral groups (Theorem~\ref{qtFCNdihedralgroups2}) and thereby prove the conjectures in this case (Corollary~\ref{co: qtfusscat dihedral groups} and Theorem~\ref{th:crystalldihedral}). Finally, we compute the $q,t$-Fu\ss-Catalan numbers for the cyclic groups as a first example of a non-real reflection group (Corollary~\ref{co: qtfusscat cyclic groups}).

		In Section~\ref{rationalcherednikalgebras}, we present some background on rational Cherednik algebras, prove a generalization of a theorem of I.~Gordon which connects the $q,t$-Fu\ss-Catalan numbers to those (Theorem~\ref{theogordon_new2}) and finally, we present a conjecture (Conjecture~\ref{surjectiontrivialrepresentation}) in this context which would imply Conjectures~\ref{conjecturedim} and \ref{conjectureqtdiag}.

	\section{Background on classical $q,t$-Fu\ss-Catalan numbers} \label{qtfusscatA}
		The symmetric group $\Sn$ acts \emph{diagonally} on the polynomial ring
		\begin{eqnarray*}
			\C[\x,\y] &:=& \C[x_1,y_1,\dots,x_n,y_n]
		\end{eqnarray*}
		by
		\begin{eqnarray}
			\sigma(x_i) := x_{\sigma(i)} &,& \sigma(y_i) := y_{\sigma(i)} \mathrm{\ for\ } \sigma \in \Sn. \label{eq: diagonal action of the symmetric group}
		\end{eqnarray}
		Note that $\C[\x,\y]$ is bigraded by degree in $\x$ and degree in $\y$ and that this diagonal action preserves the bigrading.

		The \Dfn{diagonal coinvariant ring} $\DR_n$ is defined to be $\C[\x,\y] / \mathcal{I}$, where $\mathcal{I}$ is the ideal in $\C[\x,\y]$ generated by all \Dfn{invariant polynomials without constant term}, i.e., all polynomials $p \in \C[\x,\y]$ such that $\sigma(p) = p$ for all $\sigma \in \Sn$ and $p(0) = 0$. This ring has a closely related extension for any integer $m$: let $\mathcal{A}$ be the ideal generated by all \Dfn{alternating polynomials}, i.e., all polynomials $p \in \C[\x,\y]$ such that $\sigma(p) = \sgn(\sigma) p$ for all $\sigma \in \Sn$, where $\sgn(\sigma)$ denotes the sign of the permutation $\sigma$. Then the space $\DR_n^{(m)}$ was defined by A.~Garsia and M.~Haiman in \cite{garsiahaiman} as
		\begin{eqnarray*}
			\DR_n^{(m)} := \big(\mathcal{A}^{m-1} / \mathcal{A}^{m-1} \mathcal{I} \big) \otimes \det^{\otimes (m-1)},
		\end{eqnarray*}
		where $\det$ is the $1$-dimensional \Dfn{sign representation} defined by $\sigma(z) := \sgn(\sigma) z$ for $z \in \C$ and where $\det^{\otimes k}$ is its $k$-th tensor power. As $\DR_n^{(m)}$ is a module that reduces for $m = 1$ to the diagonal coinvariant ring, we call it the \Dfn{space of generalized diagonal coinvariants}. M.~Haiman proved that the dimension of $\DR_n^{(m)}$ is equal to $(mn+1)^{n-1}$, see e.g. \cite[Theorem 1.4]{haiman7}. For $m = 1$, J.~Haglund and N.~Loehr found a conjectured combinatorial interpretation of its Hilbert series in terms of certain statistics on parking functions \cite{haglundloehr}.

		Observe that the natural $\Sn$-action on $\DR_n^{(m)}$ is twisted by the $(m-1)$-st power of the sign representation such that the generators of this module, which are the minimal generators of $\mathcal{A}^{m-1}$, become invariant. One can show that the alternating component of $\DR_n^{(m)}$ is, except for the sign-twist, naturally isomorphic to $\mathcal{A}^m / \langle \x,\y \rangle \mathcal{A}^m$, where $\langle \x, \y \rangle = \langle x_1, y_1, \dots,x_n, y_n \rangle$ is the ideal of all polynomials without constant term. Let $M^{(m)}$ denote this alternating component of $\DR_n^{(m)}$,
			\begin{eqnarray}
				M^{(m)} &:=& \eeps(\DR_n^{(m)}) \nonumber \\
								&\cong& \big(\mathcal{A}^m / \langle \x,\y \rangle \mathcal{A}^m \big) \otimes \det^{\otimes (m-1)}, \label{eq: alternating component A}
			\end{eqnarray}
			where $\eeps$ is the \Dfn{sign idempotent} defined by
			\begin{eqnarray}
				\eeps(p) := \frac{1}{n!} \sum_{\sigma \in \Sn} \sgn(\sigma) \sigma(p). \label{eq: sign idempotent}
			\end{eqnarray}
			
			This alternating component was first considered in \cite[Section 3]{garsiahaiman}, but a proof of (\ref{eq: alternating component A}) was left to the reader. It can be deduced from the following well-known lemma. We will prove the identity in a more general context in Section~\ref{qtfusscatdefinition}.

		\begin{note}
			The notions of $\det$ for the \emph{sign representation} and $\eeps$ for the \emph{sign idempotent} will become clear in Section~\ref{gendiagonalcoinvariants}, where we generalize the notions to all complex reflection groups.
		\end{note}

		\begin{lemma}[Graded version of Nakayama's Lemma]\label{lemma:nakayama}
			Let $R = \oplus_{i \geq 0} R_i$ be an $\N$-graded $k$-algebra for some field $k$ and let $M$ be a graded $R$-module, bounded below in degree. Then $\{m_1,...,m_t\}$ generate $M$ as an $R$-module if and only if their images $\{\overline{m}_1,...,\overline{m}_t\}$ $k$-linearly span the $k$-vector space $M/R_+M$, where $R_+ := \oplus_{n \geq 1} R_i$. In particular, $\{m_1,...,m_t\}$ generate M minimally as an $R$-module if and only if $\{\overline{m}_1,...,\overline{m}_t\}$ is a basis of $M/R_+M$ as a $k$-vector space.
		\end{lemma}

		\begin{remark}
			Nakayama's Lemma implies that $\mathcal{A}^m / \langle \x,\y \rangle \mathcal{A}^m$ has a vector space basis given by (the images of) any minimal generating set of $\mathcal{A}^m$ as a $\C[\x,\y]$-module. Therefore, it is often called the \emph{minimal generating space} of $\mathcal{A}^m$.
		\end{remark}

		For $X=\{ (\alpha_1,\beta_1),\dots,(\alpha_n,\beta_n) \} \subseteq \mathbb{N} \times \mathbb{N}$ define the \Dfn{bivariate Vandermonde determinant} by
		$$\Delta_X(\x,\y) := \determinant\!\left( \begin{array}{ccc} x_1^{\alpha_1}y_1^{\beta_1} & \dots & x_1^{\alpha_n}y_1^{\beta_n} \\ \vdots & & \vdots \\ x_n^{\alpha_1}y_n^{\beta_1} & \dots & x_n^{\alpha_n}y_n^{\beta_n} \end{array} \right).$$
		As a vector space, the space $\mathbb{C}[\mathbf{x},\mathbf{y}]^{\det}$ of all alternating polynomials has a well-known basis given by
		\begin{eqnarray}
			\mathcal{B} = \big\{\Delta_X : X \subseteq \mathbb{N} \times \mathbb{N}, |X|=n \big\}. \label{eq: generalized vandermonde determinant}
		\end{eqnarray}
		In particular, the ideal generated by these elements equals $\mathcal{A}$, compare \cite{haiman4}. Again by Nakayama's Lemma, $M^{(1)}$ has a vector space basis given by (the images of) any maximal linearly independent subset of $\mathcal{B}$ with coefficients in $\C[\x,\y]^{\Sn}_+$. Unfortunately, no general construction of such an independent subset is known so far.

		The $q,t$\Dfn{-Fu\ss-Catalan numbers} were first defined by M.~Haiman in \cite{haiman2} as the bigraded Hilbert series of the alternating component of the space of generalized diagonal coinvariants,
		\begin{eqnarray}
			\Cat_n^{(m)}(q,t) &:=& \mathcal{H}(M^{(m)};q,t), \label{eq: qtfusscatA}
		\end{eqnarray}
		where $\mathcal{H}(M;q,t) = \sum_{i,j \geq 0} \dim \left(M_{i,j}\right) q^i t^j$ is the \Dfn{bigraded Hilbert series} of the bigraded module $M$, and where $M_{i,j}$ denotes the bihomogeneous component of $M$ in bidegree $(i,j)$. He moreover conjectured that $\Cat_n^{(m)}(q,t)$ is in fact a $q,t$-extension of the Fu\ss-Catalan numbers $\Cat_n^{(m)}$. Using subtle results from algebraic geometry, he was finally able to prove this conjecture in the context of the $n!$- and the $(n+1)^{n-1}$-conjectures \cite{haiman3}. From this work it follows that $\Cat_n^{(m)}(q,t)$ is equal to a complicated rational function in the context of \emph{modified Macdonald polynomials}. This rational function was studied by A.~Garsia and M.~Haiman in \cite{garsiahaiman}. They were able to prove the specializations $t = 1$ and $t = q^{-1}$ in $\Cat_n^{(m)}(q,t)$. Those specializations were already conjectured by M.~Haiman in \cite{haiman2} and turn out to be equal to well-known $q$-extensions of the Fu\ss-Catalan numbers, namely the generating function for the \Dfn{area statistic} on $m$\Dfn{-Dyck paths} considered by J.~F{\"u}rlinger and J.~Hofbauer in \cite{fuerlingerhofbauer},
		\begin{eqnarray}
			\Cat_n^{(m)}(q,1) &=& \sum_{D \in \Dn^{(m)}} q^{\area(D)}, \label{eq: fuerlinger hofbauer}
		\end{eqnarray}
		and, up to a power of $q$, \Dfn{MacMahon's} $q$\Dfn{-Catalan numbers},
		\begin{eqnarray}
			q^{m\binom{n}{2}}\Cat_n^{(m)}(q,q^{-1}) &=& \frac{1}{[mn+1]_q} \left[ \begin{array}{c} (m+1)n \\ n \end{array} \right]_q. \label{eq: macmahon}
		\end{eqnarray}

		\begin{figure}
			%\centering
			\begin{tabular}[h]{cc}

\setlength{\unitlength}{.65pt}

\begin{picture}(160,160)
  \linethickness{.25\unitlength}
  \color{grau}
  \put(0  ,0  ){\line(0,1){160}}
 	\put(20 ,20 ){\line(0,1){140}}
 	\put(40 ,40 ){\line(0,1){120}}
 	\put(60 ,60 ){\line(0,1){100}}
 	\put(80 ,80 ){\line(0,1){80 }}
 	\put(100,100){\line(0,1){60 }}
 	\put(120,120){\line(0,1){40 }}
 	\put(140,140){\line(0,1){20 }}
 	
 	\put(0,20 ){\line(1,0){20 }}
 	\put(0,40 ){\line(1,0){40 }}
 	\put(0,60 ){\line(1,0){60 }}
 	\put(0,80 ){\line(1,0){80 }}
 	\put(0,100){\line(1,0){100}}
 	\put(0,120){\line(1,0){120}}
 	\put(0,140){\line(1,0){140}}
 	\put(0,160){\line(1,0){160}}
 	
 	\put(0,0){\line(1,1){160}}
 	
 	\color{schwarz}
  \linethickness{1.75\unitlength}
  
 	\put(0  ,0  ){\line(0,1){20.875}}
 	\put(0  ,20 ){\line(0,1){20.875}}
 	\put(20 ,40 ){\line(0,1){20.875}}
 	\put(20 ,60 ){\line(0,1){20.875}}
 	\put(60 ,80 ){\line(0,1){20.875}}
 	\put(80 ,100){\line(0,1){20.875}}
 	\put(80 ,120){\line(0,1){20.875}}
 	\put(100,140){\line(0,1){20.875}}
 	
 	\put(0  ,40 ){\line(1,0){20.875}}
 	\put(20 ,80 ){\line(1,0){20.875}}
 	\put(40 ,80 ){\line(1,0){20.875}}
 	\put(60 ,100){\line(1,0){20.875}}
 	\put(80 ,140){\line(1,0){20.875}}
 	\put(100,160){\line(1,0){20.875}}
 	\put(120,160){\line(1,0){20.875}}
 	\put(140,160){\line(1,0){20.875}}
 	
\end{picture}

&

\setlength{\unitlength}{.65pt}

\begin{picture}(320,160)
  \linethickness{.25\unitlength}
  \color{grau}
  \put(0  ,0  ){\line(0,1){160}}
  \put(20 ,10 ){\line(0,1){150}}
 	\put(40 ,20 ){\line(0,1){140}}
 	\put(60 ,30 ){\line(0,1){130}}
 	\put(80 ,40 ){\line(0,1){120}}
 	\put(100,50 ){\line(0,1){110}}
 	\put(120,60 ){\line(0,1){100}}
 	\put(140,70 ){\line(0,1){90 }}
 	\put(160,80 ){\line(0,1){80 }}
 	\put(180,90 ){\line(0,1){70 }}
 	\put(200,100){\line(0,1){60 }}
 	\put(220,110){\line(0,1){50 }}
 	\put(240,120){\line(0,1){40 }}
 	\put(260,130){\line(0,1){30 }}
 	\put(280,140){\line(0,1){20 }}
 	\put(300,150){\line(0,1){10 }}
 	
 	\put(0,20 ){\line(1,0){40 }}
 	\put(0,40 ){\line(1,0){80 }}
 	\put(0,60 ){\line(1,0){120}}
 	\put(0,80 ){\line(1,0){160}}
 	\put(0,100){\line(1,0){200}}
 	\put(0,120){\line(1,0){240}}
 	\put(0,140){\line(1,0){280}}
 	\put(0,160){\line(1,0){320}}
 	
 	\put(0,0){\line(2,1){320}}
 	
 	\color{schwarz}
  \linethickness{1.75\unitlength}
  
 	\put(0  ,0  ){\line(0,1){20.875}}
 	\put(0  ,20 ){\line(0,1){20.875}}
 	\put(60 ,40 ){\line(0,1){20.875}}
 	\put(60 ,60 ){\line(0,1){20.875}}
 	\put(140,80 ){\line(0,1){20.875}}
 	\put(200,100){\line(0,1){20.875}}
 	\put(200,120){\line(0,1){20.875}}
 	\put(260,140){\line(0,1){20.875}}
 	
 	\put(0  , 40){\line(1,0){20.875}}
 	\put(20 , 40){\line(1,0){20.875}}
 	\put(40 , 40){\line(1,0){20.875}}
 	\put(60 , 80){\line(1,0){20.875}}
 	\put(80 , 80){\line(1,0){20.875}}
 	\put(100, 80){\line(1,0){20.875}}
 	\put(120, 80){\line(1,0){20.875}}
 	\put(140,100){\line(1,0){20.875}}
 	\put(160,100){\line(1,0){20.875}}
 	\put(180,100){\line(1,0){20.875}}
 	\put(200,140){\line(1,0){20.875}}
 	\put(220,140){\line(1,0){20.875}}
 	\put(240,140){\line(1,0){20.875}}
 	\put(260,160){\line(1,0){20.875}}
 	\put(280,160){\line(1,0){20.875}}
	\put(300,160 ){\line(1,0){20.875}}
 	
\end{picture}

\\[10pt]
				(a) & \hspace{70pt} (b)
			\end{tabular}
			\caption{A $1$-Dyck path and a $2$-Dyck path; both have semilength $8$ and $\area 10$.}
			\label{figcatpaths}
		\end{figure}
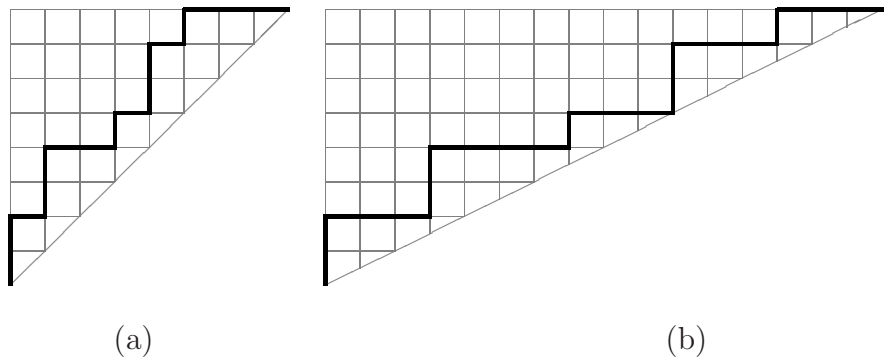	

		Here, $\Dn^{(m)}$ denotes the set of all $m$\Dfn{-Dyck paths} of semilength $n$ which are north-east lattice paths from $(0,0)$ to $(mn,n)$ that stay above the diagonal $x=my$. Moreover, the \Dfn{area} is defined to be the number of full lattice squares which lie between a path and the diagonal. See Fig.~\ref{figcatpaths} for an example.
		
		J.~Haglund defined the \emph{bounce statistic} on $1$-Dyck paths \cite{haglund2}, and N.~Loehr generalized the definition to $m$-Dyck paths \cite{loehr}. They conjectured that the $q,t$-Fu\ss-Catalan numbers can be described combinatorially in the manner of Eq.~(\ref{eq: fuerlinger hofbauer}) using the bounce statistic as the $t$-exponent. A.~Garsia and J.~Haglund were able to prove this conjecture for $m = 1$ \cite{garsiahaglund}; it remains open for $m \geq 2$.

	\section{$q,t$-Fu\ss-Catalan numbers for complex reflection groups}\label{qtfusscat}
	
		In this section, we generalize the definition of $q,t$-Fu\ss-Catalan numbers to arbitrary (finite) complex reflection groups. As we only deal with finite reflection groups, we usually suppress the term \lq finite\rq. Moreover, we present several conjectures concerning this generalization. They are based on computer experiments which are listed in the appendix. Moreover, we prove the conjectures for the dihedral groups $I_2(k) = G(k,k,2)$ and for the cyclic groups $\mathcal{C}_k = G(k,1,1)$. Here and below, $G(k,p,\ell)$ refers to the infinite family in the Shephard-Todd classification of complex reflection groups \cite{shephardtodd}.

		\subsection{The space of generalized diagonal coinvariants}\label{gendiagonalcoinvariants}

			The definition of the space of generalized diagonal coinvariants makes sense for any complex reflection group. For real reflection groups and for $m=1$, it can be found in \cite[Section 7]{haiman2}.

			Any complex reflection group $W$ of rank $\ell$ acts naturally as a matrix group on $V$ (i.e., $V$ is a \Dfn{reflection representation}) and moreover on $V \oplus V^*$ by
			$$\omega(v \oplus v^*) := \omega v \oplus {^t(\omega^{-1})}v^*.$$
			$W$ is a subgroup of the unitary group $\UV$; thus ${^t(\omega^{-1})}$ is the \emph{complex conjugate} of $\omega \in W$. This action induces a \Dfn{diagonal contragredient action} of $W$ on $V^* \oplus V$ and thereby on its symmetric algebra $S(V^*\oplus V)$ which is the ring of polynomial functions on $V \oplus V^*$. After fixing a basis for $V$, this ring of polynomial functions can be identified with
			$$\C[\x,\y]  := \C[x_1,y_1,\dots,x_\ell,y_\ell] = \C[V \oplus V^*].$$
			Observe that, as for the symmetric group, the $W$-action on $\C[\x,\y]$ preserves the bigrading on $\C[\x,\y]$ and moreover, that
			\begin{eqnarray}
				\omega(pq) = \omega(p)\omega(q) \mathrm{\ for\ all\ } p,q \in \C[\x,\y] \mathrm{\ and\ } \omega \in W. \label{eq:omegapq}
			\end{eqnarray}

			If $W$ is a real reflection group, $^t\omega = \omega^{-1}$ and $W$ therefore acts identically on $\x$ and on $\y$. In particular, this action generalizes the action described in (\ref{eq: diagonal action of the symmetric group}) for the symmetric group.
			
			Let $S$ be any $W$-module. Define the \Dfn{trivial idempotent} $\e$ to be the linear operator on $S$ defined by
			$$\e := \frac{1}{|W|}\sum_{\omega \in W}\omega  \hspace{5pt} \in \End(S),$$
			and, generalizing (\ref{eq: sign idempotent}), define the \Dfn{determinantal idempotent} $\eeps$ to be the linear operator defined by
			$$\eeps := \frac{1}{|W|}\sum_{\omega \in W}\determinant^{-1}(\omega) \hspace{2pt} \omega \hspace{5pt} \in \End(S).$$
			The trivial idempotent is a projection from $S$ onto its \Dfn{trivial component}
			$$S^W := \{ p \in S : \omega(p) = p \mathrm{\ for\ all\ } \omega \in W\},$$
			which is the isotypic component of the \Dfn{trivial representation} $\C$ defined for $\omega \in W$ by $\omega(z) := z$. Analogously, the determinantal idempotent is a projection onto its \Dfn{determinantal component}
			$$S^{\det} := \{ p \in S : \omega(p) = \determinant(\omega) \hspace*{2pt} p \mathrm{\ for\ all\ } \omega \in W\},$$
			which is the isotypic component of the \Dfn{determinantal representation} $\det$ defined for $\omega \in W$ by $\omega(z) := \determinant(\omega) \hspace*{2pt} z$. As above, $\mdet{k}$ denotes its $k$-th tensor power which is given by $\omega(z) = \determinant^{k}(\omega) \hspace*{2pt} z$. Moreover, we can define $\mdet{k}$ for negative $k$ to be the $k$-th tensor of the \Dfn{inverse determinantal representation} $\mdet{(-1)}$ defined by $\omega(z) := \determinant^{-1}(\omega) \hspace*{2pt} z$. Observe that the determinantal and the inverse determinantal representations coincide for real reflection groups, and that in this case, $\mdet{2} = \mdet{0} = \C$ is the trivial representation.
			
			$p \in S^W$ is called \Dfn{invariant in} $S$ and $p \in S^\det$ is called \Dfn{determinantal in} $S$ and we have
			\begin{eqnarray*}
				p \otimes 1^{\otimes k} \in S \otimes \mdet{k} \mathrm{\ invariant\ }	\Leftrightarrow p \otimes 1^{\otimes{(k+1)}} \in S \otimes \mdet{(k+1)} \mathrm{\ determinantal}.
			\end{eqnarray*}
			As $S \otimes \mdet{k}$ and $S \otimes \mdet{(k+1)}$ differ only by a determinantal factor in their $W$-actions, we often write $p$ instead of $p \otimes 1^{\otimes \ell}$ and say e.g. $p$ is invariant in $S \otimes \mdet{k}$ if and only if $p$ is determinantal in $S \otimes \mdet{(k+1)}$.

			\begin{definition}\label{def:generalizeddiagonalcoinvariants}
				Let $W$ be a complex reflection group of rank $\ell$ acting diagonally on $\C[\x,\y]$. Let $\mathcal{I}$ be the ideal in $\C[\x,\y]$ generated by all invariant polynomials without constant term and let $\mathcal{A}$ be the ideal generated by all determinantal polynomials. For any positive integer $m$, define the \Dfn{space of generalized diagonal coinvariants} $\DR^{(m)}(W)$ as
				$$\DR^{(m)}(W) := \big(\mathcal{A}^{m-1} / \mathcal{A}^{m-1} \mathcal{I} \big) \otimes \det^{\otimes (1-m)}.$$
			\end{definition}
			Observe that we have to twist the natural $W$-action on $\DR^{(m)}$ by the $(m-1)$-st power of the \emph{inverse determinantal representation} rather than of the determinantal representation, such that the generators of this module, which are the minimal generators of $\mathcal{A}^{m-1}$, become invariant.

			As seen above, the dimension of $\DR_n^{(m)}$ can be expressed in terms of the reflection group $\An$, which is the symmetric group $\Sn$, as
			$$\dim \DR^{(m)}(\An) = (mh+1)^\ell,$$
			where $h = n$ is the Coxeter number of $\An$ and $\ell = n-1 $ is its rank.

			\begin{figure}
				%\centering
				\begin{tabular}[h]{c|c|c}
								& $\dim$			& $(h+1)^\ell$ \\
					\hline
					$B_4$ &	$9^4+1$			& $9^4$ \\
					$B_5$ &	$11^5+33$		& $11^5$ \\
					$D_4$ &	$7^4+40$		& $7^4$
				\end{tabular}
				\caption{The actual dimension of $\DR^{(1)}(W)$ for the reflection groups $B_4,B_5$ and $D_4$.}
				\label{haimanscomputations}
			\end{figure}
	
			M.~Haiman computed the actual dimension of $\DR^{(1)}(W)$ for the reflection groups $B_4,B_5$ and $D_4$. The results can be found in Fig.~\ref{haimanscomputations}. These \lq\lq counterexamples\rq\rq\ led him to the following conjecture \cite[Conjecture 7.1.2]{haiman2}:
			\begin{conjecture}[M.~Haiman]
				For any real (or eventually crystallographic) reflection group $W$, there exists a \lq\lq natural\rq\rq\ quotient ring $R_W$ of $\C[\x,\y]$ by some homogeneous ideal containing $\mathcal{I}$ such that
				$$\dim R_W = (h+1)^\ell.$$
			\end{conjecture}

			This conjecture was proved by I.~Gordon in \cite{gordon} in the context of \emph{rational Cherednik algebras}:
			\begin{theorem}[I.~Gordon]\label{theogordon}
				Let $W$ be a real reflection group. There exists a graded $W$-stable quotient ring $R_W$ of $\DR^{(1)}(W)$ such that
				\begin{itemize}
					\item[(i)] $\dim(R_W) = (h+1)^\ell$ and moreover,
					\item[(ii)] $q^N \mathcal{H}(R_W;q) = [h+1]_q^\ell$,
				\end{itemize}
			\end{theorem}

			In Section~\ref{sec: associated graded}, we will slightly generalize this theorem to $\DR^{(m)}(W)$ for arbitrary $m \geq 1$.

%			\begin{note}
%				It follows from computations by J.~Alfano \cite{alfano} and E.~Reiner \cite{ereiner} that for the dihedral group $W = I_2(k)$, $R_W = \DR^{(1)}(W)$, see Theorem~\ref{th: alfano reiner}.
%			\end{note}

		\subsection{$q,t$-Fu\ss-Catalan numbers for complex reflection groups}\label{qtfusscatdefinition}
	
			M.~Haiman's computations of the dimension of the diagonal coinvariants in types $B_4$, $B_5$ and $D_4$ seemed to be the end of the story, but computations of the dimension of the determinantal component of the generalized diagonal coinvariants $\DR^{(m)}(W)$ suggest the following conjecture:
			\begin{conjecture}\label{conjecturedim}
				Let $W$ be a well-generated complex reflection group. Then
				\begin{eqnarray*}
					\dim \eeps \big( \DR^{(m)}(W) \big)	& =& \Cat^{(m)}(W).
				\end{eqnarray*}
			\end{conjecture}

			We used the computer algebra system {\tt Singular} \cite{singular} for the 	aforementioned computations for several classical groups including types $B_4$ and $D_4$. The computations are listed in the appendix.

			For the computations, we used the following isomorphism which was mentioned in type $A$ in Section~\ref{qtfusscatA}:
			\begin{theorem}\label{theo:altcomp}
				Let $W$ be a complex reflection group. The determinantal component of $\DR^{(m)}(W)$ is, except for a determinantal factor, naturally isomorphic (as a bigraded $W$-module) to the minimal generating space of the ideal $\mathcal{A}^m$ in $\C[\x,\y]$,
				\begin{eqnarray}
					\eeps \big( \DR^{(m)}(W) \big) &\cong& \big(\mathcal{A}^m / \langle \x,\y \rangle \mathcal{A}^m \big) \otimes \det^{\otimes (1-m)}. \label{eq: isomorphism}
				\end{eqnarray}
			\end{theorem}

			We will prove the theorem using Lemma~\ref{lemma:nakayama} (Nakayama's Lemma)\footnote{We thank Vic Reiner for improving several arguments in the proof of Theorem~\ref{theo:altcomp}.}. We also need the following simple equivalences concerning invariant and determinantal polynomials:
			\begin{lemma}\label{lemma:alt-inv}
				Let $k \in \N$ and let $W$ be a complex reflection group acting on $S := \C[\x,\y] \otimes \mdet{k}$. Let $p_i \in \C[\x,\y]$, let $\inv_i \in \C[\x,\y]$ be invariant in $S$ and let $\alt_i \in \C[\x,\y]$ be determinantal in $S$. Set $p := \sum_i p_i \alt_i$ and $q := \sum_i p_i \inv_i$. Then
				\begin{eqnarray*}
					p \mathrm{\ determinantal\ in\ } S &\Leftrightarrow& p = \sum \e(p_i) \alt_i, \\
					q \mathrm{\ determinantal\ in\ } S &\Leftrightarrow& q = \sum \eeps(p_i) \inv_i.
				\end{eqnarray*}
			\end{lemma}

			\begin{proof}
				We prove the first statement, the proof of the second is analogous. $p$ is determinantal in $S$ if and only if
				\begin{eqnarray*}
					p &=& \eeps(p) \\
						&=& \frac{1}{|W|}\sum_\omega \determinant^{-1}(\omega) \omega(\sum_i p_i \alt_i) \\
						&\stackrel{(\ref{eq:omegapq})}{=}& \frac{1}{|W|}\sum_\omega \sum_i \determinant^{-1}(\omega) \omega(p_i) \determinant(\omega) \alt_i \\
						&=& \sum_i \big( \frac{1}{|W|}\sum_\omega \omega(p_i) \big) \alt_i \\
						&=& \sum_i \e(p_i) \alt_i. \qquad \qed
				\end{eqnarray*}
			\end{proof}
	
			\begin{proof}[of Theorem \ref{theo:altcomp}]
				Using complete reducibility, one can rewrite the left-hand side of (\ref{eq: isomorphism}) as
				\begin{eqnarray}
					\eeps \big( \mathcal{A}^{m-1} \otimes \mdet{(1-m)} \big) \big/ \eeps \big(\mathcal{IA}^{m-1} \otimes \mdet{(1-m)} \big). \label{eq:complete reducibility}
				\end{eqnarray}
				Let $p \in \mathcal{A}^{m-1}$. I.e., $p$ can be written as
				$$p = \sum p_i \alt_{(i,1)} \cdots \alt_{(i,m-1)},$$
				where $p_i \in \C[\x,\y]$ and where $\alt_{(i,j)}$ is determinantal in $\C[\x,\y]$. By definition,
				$$\omega \left(\alt_{(i,1)} \cdots \alt_{(i,m-1)} \right) = \determinant^{m-1}(\omega) \alt_{(i,1)} \cdots \alt_{(i,m-1)},$$
				for $\omega \in W$, or equivalently, $\alt_{(i,1)} \dots \alt_{(i,m-1)}$ is invariant in $\C[\x,\y] \otimes \det^{\otimes(1-m)}$. Lemma~\ref{lemma:alt-inv} now implies that $p$ is determinantal in $\C[\x,\y] \otimes \det^{\otimes(1-m)}$ if and only if
				\begin{eqnarray*}
					p = \sum \eeps (p_i) \alt_{(i,1)} \dots \alt_{(i,m-1)}.
				\end{eqnarray*}
				Let $\widetilde{\mathcal{A}}^{(m)}$ be the complex vector space of all linear combinations of products of $m$ determinantal polynomials. The product of an invariant and a determinantal polynomial is again determinantal; this turns $\widetilde{\mathcal{A}}^{(m)}$ into a $\C[\x,\y]^W$-module.
				
				As $\eeps(p_i)$ is determinantal in $\C[\x,\y]$, we get that $p$ is determinantal in $\mathcal{A}^{m-1} \otimes \det^{\otimes(1-m)}$ if and only if $p \in \widetilde{\mathcal{A}}^{(m)}$. By the same argument, $p \in \mathcal{IA}^{m-1}$ is determinantal in $\mathcal{IA}^{m-1} \otimes \det^{\otimes(1-m)}$ if and only if $p \in \C[\x,\y]^W_+ \widetilde{\mathcal{A}}^{(m)} \subseteq \widetilde{\mathcal{A}}^{(m)}$.
				
				Together with (\ref{eq:complete reducibility}), we get
				\begin{eqnarray}
					\eeps \big( \DR^{(m)}(W) \big) &=& \left( \widetilde{\mathcal{A}}^{(m)} \big/ \C[\x,\y]^W_+ \widetilde{\mathcal{A}}^{(m)} \right) \otimes \mdet{1-m}. \label{eq: naka1}
				\end{eqnarray}
				By Nakayama's Lemma, the right-hand side of Eq.~(\ref{eq: naka1}) has a vector space basis given by (the images of) a minimal generating set of $\widetilde{\mathcal{A}}^{(m)}$ as a $\C[\x,\y]^W$-module. On the other hand, $\mathcal{A}^m / \langle \x,\y \rangle \mathcal{A}^m$ has a vector space basis given by (the images of) a minimal generating set of $\mathcal{A}^m$ considered as a $\C[\x,\y]$-module.
				
				$\mathcal{A}^m$ is generated as a $\C[\x,\y]$-module by all products of $m$ determinantal polynomials. Therefore, it has a minimal generating set $\mathcal{S}$ that is also contained in $\widetilde{\mathcal{A}}^{(m)}$. Using again Lemma~\ref{lemma:alt-inv}, $\mathcal{S}$ minimally generates $\widetilde{\mathcal{A}}^{(m)}$ as a $\C[\x,\y]^W$-module. Thus, the map
				$$s + \C[\x,\y]^W_+ \widetilde{\mathcal{A}}^{(m)} \mapsto s + \langle \x,\y \rangle \mathcal{A}^m$$
				for $s \in \mathcal{S}$ extends to a bigraded vector space isomorphism. As both $W$-actions coincide, this completes the proof. \qed
			\end{proof}

			\begin{definition}
				Define the $W$-module $M^{(m)}(W)$ to be the determinantal component of $\DR^{(m)}(W)$,
				\begin{eqnarray*}
				M^{(m)}(W) &:=& \eeps \big(\DR^{(m)}(W) \big) \\
										&\cong& \big(\mathcal{A}^m / \langle \x,\y \rangle \mathcal{A}^m \big) \otimes \mdet{(1-m)}.
				\end{eqnarray*}
			\end{definition}

			As we have seen in (\ref{eq: generalized vandermonde determinant}) for the symmetric group, the space $\C[\x,\y]^{\det}$ has a well-known basis given by
			$$\mathcal{B}_W := \big\{\eeps(\operatorname{m}(\x,\y)) : m(\x,\y) \mathrm{\ monomial\ in\ } \x,\y \mathrm{\ with\ } \eeps(m(\x,\y)) \neq 0 \big\},$$
			and the ideal $\mathcal{A} \subseteq \C[\x,\y]$ is generated by $\mathcal{B}_W$. Thus, finding a minimal generating set for $\mathcal{A}$ as a $\C[\x,\y]$-module is equivalent to finding a maximal linearly independent subset of $\mathcal{B}_W$ with coefficients in $\C[\x,\y]^W_+$. As for the symmetric group, it is an open problem to construct such a maximal linearly independent subset of $B_W$.

			For the other classical types, $\mathcal{B}_W$ can also be described using the bivariate Vandermonde determinant. In type $B$, it reduces to
			$$\mathcal{B}_{B_n} = \big\{\Delta_X : X = \{(\alpha_1,\beta_1),\dots,(\alpha_n,\beta_n)\} \subseteq \mathbb{N} \times \mathbb{N}, |X|=n, \alpha_i+\beta_i \equiv 1 \mod 2 \big\}$$
			and in type $D$, it reduces to
			$$\mathcal{B}_{D_n} = \big\{\Delta_X : X = \{(\alpha_i,\beta_i),\dots,(\alpha_n,\beta_n)\} \subseteq \mathbb{N} \times \mathbb{N}, |X|=n, \alpha_i+\beta_i \equiv \alpha_j+\beta_j \mod 2 \big\}.$$

			Conjecture~\ref{conjecturedim} leads to the following definition:
			\begin{definition}\label{def:qtfusscatreal}
				Let $W$ be a complex reflection group, let $\DR^{(m)}(W)$ be the space of generalized diagonal coinvariants and let $M^{(m)}(W)$ be its alternating component. Define the $q,t$\Dfn{-Fu\ss-Catalan numbers} $\Cat^{(m)}(W;q,t)$ as
					\begin{eqnarray*}
						\Cat^{(m)}(W;q,t) &:=& \mathcal{H}(M^{(m)}(W);q,t) \\
															& =& \mathcal{H}( \mathcal{A}^m / \langle \x,\y \rangle \mathcal{A}^m;q,t).
					\end{eqnarray*}
			\end{definition}

			By definition, the $q,t$-Fu\ss-Catalan numbers $\Cat^{(m)}(W;q,t)$ are polynomials in $q$ and $t$ with non-negative integer coefficients; moreover, they are, for real reflection groups, symmetric in $q$ and $t$. Conjecture~\ref{conjecturedim} would imply that
			$$\Cat^{(m)}(W;1,1) = \Cat^{(m)}(W).$$

			In the following section, we present conjectured properties of those polynomials which are based on computer experiments and which will be later supported by proving several special cases.

		\subsection{Conjectured properties of the $q,t$-Fu\ss-Catalan numbers}\label{qtfusscatconjectures}
	
			In addition to the computations of the dimension of $M^{(m)}(W)$, we computed its bigraded Hilbert series $\Cat^{(m)}(W;q,t)$ using the computer algebra system {\tt Macaulay 2} \cite{macaulay2}. The computations are as well listed in the appendix. All further conjectures are based on these computations.
	
		\subsubsection{The specializations $t^{\pm 1} = q^{\mp 1}$}

			The following conjecture, which is obviously stronger than Conjecture~\ref{conjecturedim}, would generalize Eq.~(\ref{eq: macmahon}) and would thereby answer a question of C.~Kriloff and V.~Reiner in \cite[Problem 2.2]{armstrong2}:
			\begin{conjecture}\label{conjectureqtdiag}
				Let $W$ be a well-generated complex reflection group acting on $\C[\x,\y] = \C[V \oplus V^*]$ as described above. Set $N = \sum (d_i-1)$ to be the number of reflections in $W$ and set $N^* = \sum (d^*_i+1)$ to be the number of reflecting hyperplanes. Then
				\begin{eqnarray*}
					q^{mN} \Cat^{(m)}(W;q,q^{-1}) &=& q^{mN^*} \Cat^{(m)}(W;q^{-1},q) \\
																					&=& \prod_{i=1}^\ell{\frac{[d_i+mh]_q}{[d_i]_q}}.
				\end{eqnarray*}
			\end{conjecture}
			For $W$ being a real reflection group acting on a real vector space $V$, there is a one-to-one correspondence between reflections in $W$ and reflecting hyperplanes for $W$ in $V$. Thus, the conjecture is consistent with the fact that the $q,t$-Fu\ss-Catalan numbers are symmetric in $q$ and $t$ in this case.

			This $q$-extension of the Fu\ss-Catalan numbers seems to have first appeared for real reflection groups in a paper by Y.~Berest, P.~Etingof and V.~Ginzburg \cite{BEG} where it is obtained as a certain Hilbert series. Their work implies that in this case, the extension is in fact a polynomial with non-negative integer coefficients. For well-generated complex reflection groups, this is still true, but so far it has only been verified by appeal to the classification. In Chapter~\ref{rationalcherednikalgebras}, we will exhibit the connection of the presented conjecture to the work of Y.~Berest, P.~Etingof and V.~Ginzburg.

			\begin{corollary}
				Conjecture~\ref{conjectureqtdiag} would imply that the $q$-degree of $\Cat^{(m)}(W;q,t)$ is given by $mN^*$ and the $t$-degree is given by $mN$.
			\end{corollary}

		\subsubsection{The specialization $t = 1$} \label{sec: t=1}

			By definition, $\Cat^{(m)}(W;q,t)$ is a polynomial in $\N[q,t]$. As for type $A$, this leads to the natural question of a combinatorial description of $\Cat^{(m)}(W;q,t)$: are there statistics \emph{qstat} and \emph{tstat} on objects counted by $\Cat^{(m)}(W)$ which generalize the \emph{area} and the \emph{bounce} statistics on $m$-Dyck paths $\Dn^{(m)}$ such that
			$$\Cat^{(m)}(W;q,t) = \sum_{D}{q^{\qstat(D)} t^{\tstat(D)}}?$$

			The conjecture we want to present in this section concerns the case of crystallographic reflection groups and the specialization $t=1$ of this open problem.
				
			The following definition is due to C.A.~Athanasiadis \cite{athanasiadis2} and generalizes a construction of J.-Y.~Shi \cite{shi}:
			\begin{definition}
				Let $W$ be a crystallographic reflection group acting on a real vector space $V$. The \Dfn{extended Shi arrangement} $\Shi^{(m)}(W)$ is the collection of hyperplanes in $V$ given by
				$$H_\alpha^{(k)} := \{ x : (\alpha,x) = k \} \mathrm{\ for\ } \alpha \in \Phi^+ \mathrm{\ and\ } -m < k \leq m,$$
				where $\Phi^+ \subset V$ is a set of \Dfn{positive roots} associated to $W$.
			\end{definition}
			A connected component of the complement of the hyperplanes in $\Shi^{(m)}(W)$ is called \Dfn{region} of $\Shi^{(m)}(W)$ and a \Dfn{positive region} is a region which lies in the \Dfn{fundamental chamber} of the associated \Dfn{Coxeter arrangement}, see Fig.~\ref{arrangement} for an example.
			
			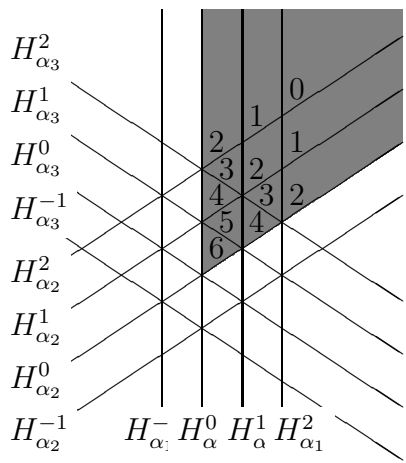
\begin{figure}
				%\centering

\setlength{\unitlength}{1.25pt}
\begin{picture}(120,140)(-60,-60)

	\color{grau}
	\multiput( 0, 0)(0.3,0.2){50}{\line(0,1){50}}
	\multiput(15,10)(0.3,0.2){50}{\line(0,1){40}}
	\multiput(30,20)(0.3,0.2){50}{\line(0,1){30}}
	\multiput(45,30)(0.3,0.2){50}{\line(0,1){20}}
	
	\multiput( 0,40)(0.0,0.2){200}{\line(1,0){60}}

	\color{schwarz}
	
	\thinlines
	\put(  0,-45){\line(0, 1){125}}
	\put(-45,-30){\line(3, 2){105}}
	\put(-45, 30){\line(3,-2){105}}

	\put(-12,-45){\line(0, 1){125}}
	\put( 12,-45){\line(0, 1){125}}
	\put( 24,-45){\line(0, 1){125}}

	\put(-45,-46){\line(3, 2){105}}
	\put(-45,-14){\line(3, 2){105}}
	\put(-45,  2){\line(3, 2){105}}

	\put(-45, 14){\line(3,-2){105}}
	\put(-45, 46){\line(3,-2){105}}
	\put(-45, 62){\line(3,-2){105}}

	\put( 2, 5){\hbox{$6$}}
	\put(14,13){\hbox{$4$}}
	\put(26,21){\hbox{$2$}}

	\put( 5,13){\hbox{$5$}}
	\put(17,21){\hbox{$3$}}

	\put( 2,21){\hbox{$4$}}
	\put(14,29){\hbox{$2$}}

	\put( 5,29){\hbox{$3$}}
	\put(26,37){\hbox{$1$}}

	\put( 2,37){\hbox{$2$}}
	\put(14,45){\hbox{$1$}}
	\put(26,53){\hbox{$0$}}

	\put(-60,-50){\colorbox{white}{$H_{\alpha_2}^{-1}$}}
	\put(-60,-34){\colorbox{white}{$H_{\alpha_2}^{ 0}$}}
	\put(-60,-18){\colorbox{white}{$H_{\alpha_2}^{ 1}$}}
	\put(-60, -2){\colorbox{white}{$H_{\alpha_2}^{ 2}$}}
	
	\put(-60, 17){\colorbox{white}{$H_{\alpha_3}^{-1}$}}
	\put(-60, 33){\colorbox{white}{$H_{\alpha_3}^{ 0}$}}
	\put(-60, 49){\colorbox{white}{$H_{\alpha_3}^{ 1}$}}
	\put(-60, 65){\colorbox{white}{$H_{\alpha_3}^{ 2}$}}

	\put(-26,-50){\colorbox{white}{$H_{\alpha_1}^{-1}$}}
	\put(-10,-50){\colorbox{white}{$H_{\alpha_1}^{ 0}$}}
	\put(  5,-50){\colorbox{white}{$H_{\alpha_1}^{ 1}$}}
	\put( 19,-50){\colorbox{white}{$H_{\alpha_1}^{ 2}$}}

\end{picture}

				\caption{The extended Shi arrangement $\Shi^{(2)}(A_2)$.}
				\label{arrangement}
			\end{figure}
			
			The following result concerning the total number of regions of $\Shi^{(m)}(W)$ had been conjectured by P.~Edelman and V.~Reiner \cite[Conjecture 3.3]{edelmanreiner}, and by C.A.~Athanasiadis \cite[Question 6.2]{athanasiadis1} and was proved uniformly by M.~Yoshinaga in \cite[Theorem 1.2]{yoshinaga}:
			\begin{theorem}[M.~Yoshinaga]
				Let $W$ be a crystallographic reflection group and let $m$ be a positive integer. Then the number of regions of $\Shi^{(m)}(W)$ is equal to $(mh+1)^\ell$,
				where $\ell$ is the rank of $W$ and where $h$ is its Coxeter number.
			\end{theorem}
			In \cite{athanasiadis3}, C.A.~Athanasiadis counted the number of positive regions of $\Shi^{(m)}(W)$:
			\begin{theorem}[C.A.~Athanasiadis]
				Let $W$ be a crystallographic reflection group and let $m$ be a positive integer. Then
				\begin{eqnarray*}
					\left|\left\{ \mathrm{positive\ regions\ of\ } \Shi^{(m)}(W) \right\}\right| = \Cat^{(m)}(W).
				\end{eqnarray*}
			\end{theorem}

			Let $W$ be a crystallographic reflection group. Fix the positive region $R^0$ to be the region given by $\{ x : 0 < (\alpha,x) < 1 \mathrm{\ for\ all\ } \alpha \in \Phi^+ \}$, and the positive region $R^\infty$ to be the region given by $\{ x : (\alpha,x) > m \mathrm{\ for\ all\ } \alpha \in \Phi^+ \}$.
			In terms of affine reflection groups, $R^0$ is called \Dfn{fundamental alcove}. The \Dfn{height} of a region is defined to be the number of hyperplanes in $\Shi^{(m)}(W)$ that separate $R$ from $R^0$ and the \Dfn{coheight} of a region $R$, denoted by $\coh(R)$, is defined by
			$$\coh(R) := mN-\operatorname{height}(R),$$
			where $N$ denotes the number of positive roots and of reflecting hyperplanes. Observe that the coheight counts, for a positive region $R$, the number of hyperplanes separating $R$ from $R^\infty$.

			\begin{conjecture}\label{conjectureqtq}
				Let $W$ be a crystallographic reflection group. Then the $q,t$-Fu\ss-Catalan numbers reduce for the specialization $t=1$ to
				$$\Cat^{(m)}(W;q,1) = \Cat^{(m)}(W;q) := \sum_R{q^{\coh(R)}},$$
				where the sum ranges over all positive regions of $\Shi^{(m)}(W)$.
			\end{conjecture}
			
			For example, it can be seen in Fig.~\ref{arrangement} that
			$$\Cat^{(2)}(A_2;q) = 1 + 2q + 3q^2 + 2q^3 + 2q^4 + q^5 + q^6.$$
			This is equal to the specialization $t = 1$ in 
			$$\Cat^{(2)}(A_2;q,t) = q^6 + q^5t + \dots + qt^5 + t^6 + q^4t + \dots + qt^4 + q^2t^2.$$
			\begin{proposition} \label{prop:AnmDyck}
				Let $W = \An$. Then $\Cat^{(m)}(W;q)$ is equal to the area generating function on $m$-Dyck paths.
			\end{proposition}

			To prove the proposition, we define filtered chains in the root poset associated to $W$. They were introduced by C.A.~Athanasiadis in \cite{athanasiadis3}. Define a partial order on a set of positive roots $\Phi^+$ associated to $W$ by letting $\alpha < \beta$ if $\beta - \alpha$ is a non-negative linear combination of simple roots. Equipped with this partial order, $\Phi^+$ is called the {\sf root poset} associated to $W$; it does not depend on the specific choice of positive roots. Let $\mathcal{I} = \{ I_1 \subseteq \ldots \subseteq I_m \}$ be an increasing chain of \Dfn{order ideals} in $\Phi^+$ (i.e., $\alpha \leq \beta \in I_i \Rightarrow \alpha \in I_i$). $\mathcal{I}$ is called a {\sf filtered chain} of length $m$ if
			$$(I_i + I_j) \cap \Phi^+ \subseteq I_{i+j}$$
			holds for all $i,j \geq 1$ with $i+j \leq m$, and
			$$(J_i + J_j) \cap \Phi^+ \subseteq J_{i+j}$$
			holds for all $i,j \geq 1$, where $J_i := \Phi^+ \setminus I_i$ and $J_i = J_m$ for $i > m$. C.A.~Athanasiadis constructed an explicit bijection $\psi$ between positive regions of $\Shi^{(m)}(W)$ and filtered chains of length $m$ in $\Phi^+$, such that
				\begin{eqnarray*}
					\coh(R) = |\psi(R)|,
				\end{eqnarray*}
				where $|\{I_1 \subseteq \dots \subseteq I_m\}| := |I_1| + \dots + |I_m|$. In particular, this implies that
				$$\Cat^{(m)}(W;q) = \sum_{\mathcal{I}} q^{|\mathcal{I}|},$$
			where the sum ranges over all filtered chains of length $m$ in $\Phi^+$.

			\begin{figure}
				%\centering
				
\setlength{\unitlength}{0.8pt}

\begin{picture}(400,60)

% two $1$-Dyck paths 

% first $1$-Dyck path

  \linethickness{.25\unitlength}
  \color{grau}
  \put(120, 0 ){\line(0,1){60}}
  \put(140,20 ){\line(0,1){40}}
 	\put(160,40 ){\line(0,1){20}}
 	
 	\put(120,20 ){\line(1,0){20}}
 	\put(120,40 ){\line(1,0){40}}
 	\put(120,60 ){\line(1,0){60}}
 	
 	\put(120,0){\line(1,1){60}}
 	
 	\color{schwarz}
  \linethickness{1.25\unitlength}
  
 	\put(120,0 ){\line(0,1){20.625}}
 	\put(120,20){\line(0,1){20.625}}
 	\put(160,40){\line(0,1){20.625}}
 	
 	\put(120,40){\line(1,0){20.625}}
 	\put(140,40){\line(1,0){20.625}}
 	\put(160,60){\line(1,0){20.625}}

	\put(185,25){$+$}
% second $1$-Dyck path

  \linethickness{.25\unitlength}
  \color{grau}
  \put(200, 0 ){\line(0,1){60}}
  \put(220,20 ){\line(0,1){40}}
 	\put(240,40 ){\line(0,1){20}}
 	
 	\put(200,20 ){\line(1,0){20}}
 	\put(200,40 ){\line(1,0){40}}
 	\put(200,60 ){\line(1,0){60}}
 	
 	\put(200,0){\line(1,1){60}}
 	
 	\color{schwarz}
  \linethickness{1.25\unitlength}
  
 	\put(200,0 ){\line(0,1){20.625}}
 	\put(200,20){\line(0,1){20.625}}
 	\put(200,40){\line(0,1){20.625}}
 	
 	\put(200,60){\line(1,0){20.625}}
 	\put(220,60){\line(1,0){20.625}}
 	\put(240,60){\line(1,0){20.625}}

	\put(265,25){$=$}
% one $2$-Dyck path

  \linethickness{.25\unitlength}
  \color{grau}
  \put(280,0 ){\line(0,1){60}}
  \put(300,10){\line(0,1){50}}
 	\put(320,20){\line(0,1){40}}
 	\put(340,30){\line(0,1){30}}
 	\put(360,40){\line(0,1){20}}
 	\put(380,50){\line(0,1){10}}
 	
 	\put(280,20){\line(1,0){40 }}
 	\put(280,40){\line(1,0){80 }}
 	\put(280,60){\line(1,0){120}}
 	
 	\put(280,0){\line(2,1){120}}
 	
 	\color{schwarz}
  \linethickness{1.25\unitlength}
  
 	\put(280,0 ){\line(0,1){20.625}}
 	\put(280,20){\line(0,1){20.625}}
 	\put(320,40){\line(0,1){20.625}}
 	
 	\put(280,40){\line(1,0){20.625}}
 	\put(300,40){\line(1,0){20.625}}
 	\put(320,60){\line(1,0){20.625}}
 	\put(340,60){\line(1,0){20.625}}
 	\put(360,60){\line(1,0){20.625}}
 	\put(380,60){\line(1,0){20.625}}

% filtered chain of order ideals

	\linethickness{0.25\unitlength}
 	
 	\put(90,28){\hbox{$\mapsto$}}

 	\color{grau}
 	\put( 0,20){\circle*{10}}
 	
 	\put(55,20){\circle*{10}}
 	\put(65,40){\circle*{10}}
 	\put(75,20){\circle*{10}}
 	 	
 	\color{schwarz} 
 	\put( 0,20){\line(1, 2){10}}
 	\put(10,40){\line(1,-2){10}}
 	
 	\put(33,28){\hbox{$\subseteq$}}

 	\put(55,20){\line(1, 2){10}}
 	\put(65,40){\line(1,-2){10}}

 	\put( 0,20){\circle*{5}}
	\put(-5,10){{\tiny $12$}}
	\put(10,40){\circle*{5}}
	\put(4.5,47){{\tiny $13$}}
 	\put(20,20){\circle*{5}}
	\put(15,10){{\tiny $23$}}
 	
 	\put(55,20){\circle*{5}}
	\put(50,10){{\tiny $12$}}
 	\put(65,40){\circle*{5}}
	\put(59.5,47){{\tiny $13$}}
 	\put(75,20){\circle*{5}}
	\put(70,10){{\tiny $23$}}

\end{picture}
				
				\caption{A filtered chain of order ideals in the root poset of type $A_2$ where the root $e_j-e_i$ for $1 \leq i < j \leq 3$ is denoted by $ij$, and the associated $2$-Dyck path $(0,1,0) + (0,1,2) = (0,2,2)$.}
				\label{m-Catpath}
			\end{figure}
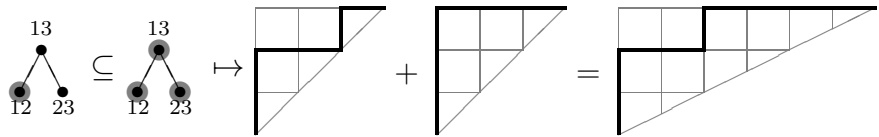

			\begin{proof}[of Proposition~\ref{prop:AnmDyck}]
				To prove the proposition, we construct a bijection between $m$-Dyck paths of semilength $n$ and filtered chains of length $m$ in the root poset $\Phi^+ = \{ \epsilon_j - \epsilon_i : 1 \leq i < j \leq n\}$ of type $\An$. An $m$-Dyck path of semilength $n$ can be encoded as a sequence $(a_1,\dots,a_n)$ of integers such that $a_1=0$ and $a_{i+1} \leq a_i+m$. Define the sum of an $m$-Dyck path and an $m'$-Dyck path, both of semilength $n$, to be the $(m+m')$-Dyck path of semilength $n$ obtained by adding the associated sequences componentwise,
				$$(a_1,\dots,a_n) + (a'_1,\dots,a'_n) := (a_1+a'_1,\dots,a_n+a'_n).$$
				Together with the well-known bijection between order ideals in $\Phi^+$ and $1$-Dyck paths, this yields a map from filtered chains of length $m$ in $\Phi^+$ to $m$-Dyck paths by summing the $1$-Dyck paths associated to the ideals in the filtered chain; see Fig.~\ref{m-Catpath} for an example. Observe that this map sends the coheight of a given filtered chain to the area of the associated $m$-Dyck path.

				To show that this map is in fact a bijection, let $\{I_1 \subseteq \dots \subseteq I_m\}, \{I'_1 \subseteq \dots \subseteq I'_m\}$ be two filtered chains which map to the same path. Assume that they are not equal, i.e., there exists $\epsilon_j-\epsilon_i \in I_\ell \setminus I'_\ell$ for some $\ell$. As both chains map to the same path, there exists $k < i, \ell' > 0$, such that $\epsilon_j-\epsilon_k \in I'_{\ell+\ell'} \setminus I_{\ell+\ell'}$. As both chains are filtered, this gives $\epsilon_i-\epsilon_k \in I'_{\ell'} \setminus I_{\ell'}$. This gives rise to an infinite sequence $(\epsilon_j-\epsilon_i,\epsilon_i-\epsilon_k,\dots)$ of pairwise different positive roots, which is a contradiction. As it is known that both sets have the same cardinality, the statement follows. \qed
			\end{proof}

		\subsection{The dihedral groups}
	
			In unpublished work in the context of their PhD-theses \cite{alfano,ereiner}, J.~Alfano and E.~Reiner were able to describe uniformly the diagonal coinvariant ring $\DR^{(1)}(W)$ for $W$ being a dihedral group. For the sake of more readability, we introduce the $q,t$-extension $[n]_{q,t}$ of the integer $n$ which we define by
			$$[n]_{q,t} := \frac{q^n - t^n}{q-t} = q^{n-1} + q^{n-2}t + \dots + qt^{n-2} + t^{n-1}.$$
			Then $[n]_{q,1} = [n]_{1,q} = [n]_q$ is the well-known $q$-extension of the integer $n$. The following description is taken from \cite[Section 7.5]{haiman2}:
			\begin{theorem}[J.~Alfano, E.~Reiner]\label{th: alfano reiner}
				Let $W = I_2(k)$ be the dihedral group of order $2k$. Then
				\begin{eqnarray*}
					\mathcal{H}(\DR^{(1)}(W);q,t) &=& 1 + [k+1]_{q,t} + qt + 2\sum_{i = 1}^{k-1} [i+1]_{q,t}.
				\end{eqnarray*}
			\end{theorem}
			
			By a simple computation, we get the following corollary.
			\begin{corollary}
				Let $N = k$ be the number of reflections in $W = I_2(k)$ and let $h = k$ be its Coxeter number. Then
				$$q^N \mathcal{H}(\DR^{(1)}(W);q,q^{-1}) = [h+1]_q^2.$$
				In particular, the quotient in Theorem~\ref{theogordon} is trivial for the dihedral groups, $R_W = \DR^{(1)}(W)$.
			\end{corollary}
			J.~Alfano and E.~Reiner obtained Theorem~\ref{th: alfano reiner} by providing an explicit description of $\DR^{(1)}(W)$:
			\begin{itemize}
				\item[(i)] the first $1$ belongs to the unique copy of the trivial representation in bidegree $(0,0)$,
				\item[(ii)] the string $[k+1]_{q,t} + qt$ belongs to copies of the determinantal representation which are generated by
					$$D,\Delta(D),\dots,\Delta^k(D) \mathrm{\ and\ } x_1y_2-x_2y_1,$$
					where
					$$D(x_1,x_2) := 2k \prod_{i = 0}^{k-1} (\sin(\pi i/k) x_1 + \cos(\pi i/k) x_2)$$
					is the \Dfn{discriminant} of $W$ and $\Delta$ is the operator defined by $\Delta := \partial_{x_1} \cdot y_1 + \partial_{x_2} \cdot y_2$, and where $\Delta^\ell(D)$ has bidegree $(k-\ell,\ell)$, and
				\item[(iii)] the later sum belongs to $\mathfrak{sl}_2$-strings.
			\end{itemize}
			By Theorem~\ref{th: alfano reiner} and the following discussion, we can immediately compute the $q,t$-Fu\ss-Catalan numbers for the dihedral groups:
			\begin{corollary} \label{cor:qtqdihedral}
				Let $W = I_2(k)$. Then
				\begin{eqnarray*}
					\Cat^{(1)}(W;q,t) &=& [k+1]_{q,t} + qt \\
														&=& q^k + q^{k-1}t + \dots + qt^{k-1} + t^k + qt.
				\end{eqnarray*}
			\end{corollary}

			In \cite[Chapter 5.4.1]{armstrong1}, D.~Armstrong suggests, how the \lq root poset\rq\ for the dihedral group $I_2(k)$ \emph{should} look like for any $k$. For the crystallographic dihedral groups, it reduces to the root poset introduced in Section~\ref{sec: t=1}. We reproduce his suggestion in Fig.~\ref{rootposet}.
			
			\begin{figure}
				%\centering
				
\setlength{\unitlength}{0.75pt}
\begin{picture}(40,100)
	\linethickness{.5\unitlength}

	\put( 0, 0){\line(1, 1){20}}
	\put(20,20){\line(1,-1){20}}
	\put( 0, 0){\circle*{5}}
	\put(40, 0){\circle*{5}}
	\put(20,20){\circle*{5}}
	
	\put(20,40){\dashbox{5}(0,40){}}
	
	\put(20, 20){\line(0, 1){20}}
	\put(20, 80){\line(0, 1){20}}
	\put(20, 40){\circle*{5}}
	\put(20, 80){\circle*{5}}
	\put(20,100){\circle*{5}}
\end{picture}

				\caption{D.~Armstrong's suggestion for a root poset of type $I_2(k)$.}
				\label{rootposet}
			\end{figure}
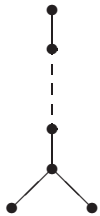

			\begin{corollary}\label{qtFCNdihedralgroup}
				Let $\Phi^+$ be the \lq root poset\rq\ associated to the dihedral group $I_2(k)$ as shown in Fig.~\ref{rootposet}. Then
				$$\Cat^{(1)}(I_2(k);q,1) = \sum_I q^{\coh(I)},$$
				where the sum ranges over all order ideals in $\Phi^+$. In particular, Conjecture~\ref{conjectureqtq} holds for the crystallographic dihedral groups $I_2(k)$ with $k \in \{2,3,4,6\}$ and $m = 1$.
			\end{corollary}

			From Theorem~\ref{th: alfano reiner}, one can also deduce the $q,t$-Fu\ss-Catalan numbers $\Cat^{(m)}(W;q,t)$.
			\begin{theorem}\label{qtFCNdihedralgroups2}
				The $q,t$-Fu\ss-Catalan numbers for the dihedral group $I_2(k)$ are given by
				$$\Cat^{(m)}(I_2(k);q,t) = \sum_{j = 0}^m q^{m-j}t^{m-j}[jk+1]_{q,t}.$$
			\end{theorem}

			To prove the theorem, we need the following lemma:
			\begin{lemma}\label{propqtfusscatdihedral}
				Let $k$ be a positive integer and let $p \in \C[x_1,x_2]$ be homogeneous of degree $a+b$. Let $0 \leq i_1,\dots,i_k,j_1,\dots,j_k \leq a+b$ be two sequences such that $\sum i_\ell = \sum j_\ell$. Then 
				$$\Delta^{j_1}(p) \cdots \Delta^{j_{k}}(p) \in \left\langle \Delta^{i_1}(p) \cdots \Delta^{i_{k}}(p), B^{k-1}\cdot(x_1y_2-x_2y_1) \right\rangle,$$
				where $B := \left\{ p,\Delta(p),\dots,\Delta^{a+b}(p),x_1y_2-x_2y_1 \right\}$.
			\end{lemma}

			\begin{proof}
				Let $\m = x_1^a x_2^b$ be a monomial in $p$ and let $i \leq a+b$. By definition,
				$$\Delta^i(\m) = i! \sum_{\ell = 0}^i \binom{a}{\ell}\binom{b}{i-\ell} x_1^{a-\ell}y_1^\ell x_2^{b-i+\ell} y_2^{i-\ell},$$
				and therefore,
				\begin{eqnarray*}
					\Delta^i(\m) \equiv \Delta_{\mod}^i(\m) := c_i\ x_2^{-i} y_2^{i}\ \m \quad \mod\ \langle x_1y_2-x_2y_1 \rangle,
				\end{eqnarray*}
				where $c_i := i! \sum_{\ell = 0}^i \binom{a}{\ell}\binom{b}{i-\ell} = i! \binom{a+b}{i}$. Observe that $c_i$ does only depend on $i$ and on $a+b$ and that $c_i > 0$. Hence, as $p$ is homogeneous of degree $a+b$, the linearity of $\Delta$ implies
				\begin{eqnarray*}
					\Delta^i(p) \equiv \Delta_{\mod}^i(p) := c_i\ x_2^{-i} y_2^{i}\ p \quad \mod\ \langle x_1y_2-x_2y_1 \rangle.
				\end{eqnarray*}
				Moreover, we have
				\begin{eqnarray*}
					\Delta^{i_1}(p) \cdots \Delta^{i_k}(p) \equiv \Delta^{i_1}(p) \cdots \Delta^{i_{k-1}}(p) \Delta_{\mod}^{i_k}(p) \quad \mod\ \langle B^{k-1} \cdot (x_1y_2-x_2y_1)\rangle.
				\end{eqnarray*}
				With $\Delta^i(p)$ and $x_1y_2-x_2y_1$, $\Delta_{\mod}^i(p)$ is as well contained in $\langle B \rangle$ and therefore,
				\begin{eqnarray*}
					\Delta^{i_1}(p) \cdots \Delta^{i_k}(p) &\equiv&	\Delta_{\mod}^{i_1}(p) \cdots \Delta_{\mod}^{i_k}(p) \quad \mod\ \langle B^{k-1} \cdot (x_1y_2-x_2y_1) \rangle.
				\end{eqnarray*}
				Setting $d := c_{i_1} \cdots c_{i_k}$, the right-hand side equals $d x_2^{-\sum{i_\ell}} y_2^{\sum{i_\ell}} p^k$. By the same argument,
				\begin{eqnarray*}
					\Delta^{j_1}(p) \cdots \Delta^{j_k}(p) \equiv c x_2^{-\sum{j_\ell}} y_2^{\sum{j_\ell}} p^k \quad \mod\ \langle B^{k-1} \cdot (x_1y_2-x_2y_1)\rangle,
				\end{eqnarray*}
				where $c := c_{j_1} \cdots c_{j_k}$. As $\sum{i_\ell} = \sum{j_\ell}$, we obtain
				\begin{eqnarray*}
					d \Delta^{j_1}(p) \cdots \Delta^{j_k}(p) - c \Delta^{i_1}(p) \cdots \Delta^{i_k}(p) \in \langle B^{k-1} \cdot (x_1y_2-x_2y_1)\rangle.
				\end{eqnarray*}
				As $c,d \neq 0$, the lemma follows. \qed
			\end{proof}

			\begin{proof}[of Theorem~\ref{qtFCNdihedralgroups2}]
				Recall that the ideal $\mathcal{A}^m$ is generated by all products of $m$ generators of determinantal representations in the various bidegrees. We have seen above that $\mathcal{A}$ is minimally generated by
				$$B := \big\{ D,\Delta(D),\dots,\Delta^k(D), x_1y_2-x_2y_1 \big\},$$
				and that $\Delta^\ell(D)$ has bidegree $(k-\ell,\ell)$ and $x_1y_2-x_2y_1$ has bidegree $(1,1)$. For a given $0 \leq j \leq m$ and $0 \leq i \leq k(m-j)$, the generators of $\mathcal{A}^m$ having bidegree $((m-j)k-i+j,i+j)$ are of the form
				\begin{eqnarray}
					\Delta^{i_1}(D) \cdots \Delta^{i_{m-j}}(D)\cdot(x_1y_2-x_2y_1)^j, \label{eq:Deltas}
				\end{eqnarray}
				where $\sum i_\ell = i$. As $D$ is homogeneous in $\x$ of degree $k$, the previous lemma implies that any subset of $B^m$ which minimally generates $\mathcal{A}^m$ contains one and only one generator of the form (\ref{eq:Deltas}) for each $0 \leq j \leq m$ and each $0 \leq i \leq k(m-j)$. Thus,
				\begin{eqnarray*}
					\Cat^{(m)}(I_2(k);q,t)	&=& \sum_{j = 0}^m \sum_{i = 0}^{k(m-j)} q^{(m-j)k-i+j}t^{i+j} \\
																	&=& \sum_{j = 0}^m \sum_{i = 0}^{kj} q^{jk-i+(m-j)}t^{i+(m-j)} \\
																	&=& \sum_{j = 0}^m q^{m-j}t^{m-j} [kj+1]_{q,t}.\quad \quad \qed
				\end{eqnarray*}
			\end{proof}

			From this theorem, we can immediately deduce the following recurrence relation:
			\begin{corollary}\label{corollaryqtfusscatrecurrence}
				The $q,t$-Fu\ss-Catalan numbers for the dihedral group $I_2(k)$ satisfy the recurrence relation
				$$\Cat^{(m)}(W;q,t) = [mk+1]_{q,t} + qt\Cat^{(m-1)}(W;q,t).$$
			\end{corollary}

			We can also deduce Conjecture~\ref{conjecturedim} and Conjecture~\ref{conjectureqtdiag} for the dihedral groups:
			\begin{corollary}\label{co: qtfusscat dihedral groups}
				Conjectures~\ref{conjecturedim} and \ref{conjectureqtdiag} hold for the dihedral groups: let $W = I_2(k)$, then
				$$q^{mk}\Cat^{(m)}(W;q,q^{-1}) = \frac{[2+mk]_q [k+mk]_q}{[2]_q[k]_q}.$$
			\end{corollary}

			\begin{proof}
				By Theorem~\ref{qtFCNdihedralgroups2}, we have
				\begin{eqnarray*}
					q^{mk}\Cat^{(m)}(W;q,q^{-1})  &=& q^{mk} \sum_{j = 0}^m [jk+1]_{q,q^{-1}} = \sum_{j = 0}^m q^{(m-j)k}\frac{[2jk+2]_q}{[2]_q}.
				\end{eqnarray*}
				Therefore, it remains to show that
				\begin{eqnarray}
					\sum_{j = 0}^m q^{(m-j)k}[2jk+2]_q &=& \frac{[2+mk]_q [k+mk]_q}{[k]_q}. \label{eq:qqFCNdihedralgroups}
				\end{eqnarray}
				To prove this equality observe that on the left-hand side of (\ref{eq:qqFCNdihedralgroups}), the terms for $j$ and $m-j$ sum up to $[mk+2]_q (q^{jk} + q^{(m-j)k})$. This gives
				\begin{eqnarray*}
					\sum_{j = 0}^m q^{(m-j)k}[2jk+2]_q &=& [mk+2]_q (1 + q^k + \dots + q^{mk}). \qquad \qed
				\end{eqnarray*}
			\end{proof}

			Our next goal is to generalize Corollary~\ref{qtFCNdihedralgroup} for the crystallographic dihedral groups.

			\begin{theorem} \label{th:crystalldihedral}
				Conjecture~\ref{conjectureqtq} holds for the dihedral group $I_2(k)$ with $k \in \{2,3,4,6\}$.
			\end{theorem}
			To prove the theorem, we define a \Dfn{region walk} from $R^\infty$ to $R^0$ to be a sequence of consecutive regions $R^\infty = R_0,\dots,R_i=R^0$ of $\Shi^{(m)}(W)$ such that $\coh(R_\ell) = \ell$ and $R_\ell$ and $R_{\ell+1}$ are seperated by exactly one hyperplane.
			
			First, we prove the case $B_2 = I_2(4)$, the cases $k \in \{2,3\}$ are analogous.
			\begin{proof}[for $k = 4$]

				Our goal is to show that $\Cat^{(m)}(B_2,q)$ satisfy the recurrence relation in Corollary~\ref{corollaryqtfusscatrecurrence} for $t = 1$,
				$$\Cat^{(m)}(B_2;q) = [4m+1]_q + q\Cat^{(m-1)}(B_2;q),$$
				the theorem then follows with Corollary~\ref{cor:qtqdihedral}.

%				First, we construct an embedding of filtered chains of length $m-1$ in the root poset $\Phi^+ = \{e_2\pm e_1,e_1,e_2\}$ associated to $W$ into filtered chains of length $m$, such that the number of elements increases by $1$: let $\mathcal{I} = \{I_1 \subseteq \dots \subseteq I_{m-1}\}$ be such a filtered chain. Define $\tilde \mathcal{I} = \{\tilde I_0 \subseteq \tilde I_1 \subseteq \dots \subseteq \tilde I_{m-1}\}$ by first setting $\tilde I_0 := \emptyset, \tilde I_k := I_k$ for $k \geq 1$, and then inserting the positive root $e_1$ into the unique order ideal $\tilde I_k$ such that $e_1 \in I_{k+1}$ but $e_1 \notin I_k$. Assume that there exist $\alpha \notin \tilde I_i, \beta \notin \tilde I_j$, but $\alpha + \beta \in \tilde I_{i+j+1}$. Then $\alpha \notin I_i, \beta \notin I_j$, and as $e_1$ is a simple root, $\alpha + \beta \in I_{i+j+1}$. As $\mathcal{I}$ is filtered, this implies $\alpha \in I_{i+1}, \beta \in I_{j+1}$. As $\alpha = e_1$ or $\beta = e_1$, this gives $\alpha \in \tilde I_i$ or $\beta \in \tilde I_j$, a contradiction. Now, assume that there exist $\alpha \in \tilde I_i, \beta \in \tilde I_j$, but $\alpha + \beta \notin \tilde I_{i+j+1}$. Then $\alpha + \beta \notin I_{i+j+1}$. As $\alpha \not = e_1$ or $\beta \not = e_1$, this gives $\alpha \in I_i$ or $\beta \in I_j$, and moreover, $\beta \notin I_{j+1}$ or $\alpha \notin I_{i+1}$, which is again a contradiction. In total, this implies that $\tilde \mathcal{I}$ is in fact filtered.

			\begin{figure}
				%\centering
				\begin{tabular}{ccc}
						
\setlength{\unitlength}{2.7pt}
\begin{picture}(41,43)(-1,-3)

	\color{black}
	\thinlines
	
	\put(0,-1){\line(0,1){41}}
	\multiput(10,0)(10,0){3}{\line(0,1){40}}

	\put(-1,-4){\scriptsize $H_{e_2-e_1}^{(0)}$}

	\put(-1,0){\line(1,0){41}}
	\multiput(0,10)(0,10){3}{\line(1,0){40}}

	\put(-7, 0){\scriptsize $H_{e_1}^{(0)}$}
	
	\put(-0,10){\line(1,-1){10}}
	\put(-0,20){\line(1,-1){20}}
	\put(-0,30){\line(1,-1){30}}

	\put(-0, 5){\line(2,-1){10}}
	\put(-0,10){\line(2,-1){20}}
	\put(-0,15){\line(2,-1){30}}

\end{picture}
						
&  &

\setlength{\unitlength}{2.7pt}
\begin{picture}(57,53)(-7,-3)

	\color{grau}
	\thicklines
	
	\put(5 , 1){\line(0, 1){44}}
	\put(5 ,45){\line(1, 0){40}}
	
	\multiput(10.25, 0)(0.25,0){160}{\line(0,1){40}}

	\color{black}
	\thinlines
	
	\qbezier(47,48)(52,51)(55,46)
	\put(47.5,48.25){\vector(-1,-1){1.5}}
	
	\qbezier(2,1)(-4,2)(-6,5)
	\put(2,1){\vector(1,0){1.5}}
	
	\put(55,43){$R^\infty$}
	\put(-9,6){$R^0$}

	\put(0,-1){\line(0,1){51}}
	\multiput(10,0)(10,0){4}{\line(0,1){50}}

	\put(-1,-4){\scriptsize $H_{e_2-e_1}^{(0)}$}

	\put(-1,0){\line(1,0){51}}
	\multiput(0,10)(0,10){4}{\line(1,0){50}}

	\put(-8,-1.5){\scriptsize $H_{e_1}^{(0)}$}
	
	\put(-0,10){\line(1,-1){10}}
	\put(-0,20){\line(1,-1){20}}
	\put(-0,30){\line(1,-1){30}}
	\put(-0,40){\line(1,-1){40}}

	\put(-0, 5){\line(2,-1){10}}
	\put(-0,10){\line(2,-1){20}}
	\put(-0,15){\line(2,-1){30}}
	\put(-0,20){\line(2,-1){40}}

\end{picture}

\\
						(a) && (b)
				\end{tabular}
				\caption{(a) The positive regions of $\Shi^{(3)}(B_2)$. (b) Their embedding into the positive regions of $\Shi^{(4)}(B_2)$; the remaining region walk from $R^\infty$ to $R^0$.}
				\label{arrangementB24}
			\end{figure}
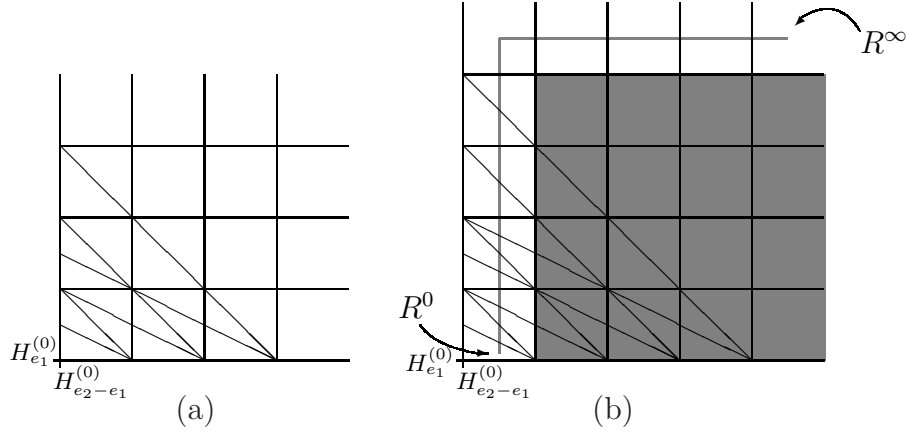

			As indicated in Fig.~\ref{arrangementB24}, it is immediate how we can embed the positive regions of $\Shi^{(m-1)}(B_2)$ into the positive regions of $\Shi^{(m)}(B_2)$; note that this embedding increases the coheight by $1$. The remaining regions form a region walk from $R^\infty$ to $R^0$ and $\sum_R q^{\coh(R)} = [4m+1]_q$, where the sum ranges over all regions in the region walk. This gives the proposed recurrence relation. \qed

%				Obviously, a filtered chain $\mathcal{I} = \{I_1 \subseteq \dots \subseteq I_{m}\}$ of length $m$ does not come from a filtered chains of length $m-1$ in the sense described above if and only if either $e_1 \notin I_m$, or $e_2-e_1 \in I_1$. In terms of the exteded Shi arrangement, those form a \emph{path} from the region associated with the empty filtered chain $\{\emptyset \subseteq \dots \subseteq \emptyset\}$ to the region associated to the full filtered chain $\{\Phi^+ \subseteq \dots \subseteq \Phi^+\}$, see Fig.~\ref{arrangementB24} for an example: first, cross all $m$ hyperplanes corresponding to $e_2-e_1$ arriving in the region associated to the filtered chain $\{\{e_2-e_1\} \subseteq \dots \subseteq \{e_2-e_1\}\}$. Next, cross a hyperplane corresponding to $e_1$, then one corresponding to $e_2$, then again one corresponding to $e_1$ and so on, crossing $m+1$ hyperplanes in total (ending after crossing a hyperplane corresponding to $e_1$ if $m$ is even and corresponding to $e_2$ if $m$ is odd). The last sequence of hyperplanes depends on the parity of $m$: if $m$ is even, then cross a hyperplane corresponding to $e_2+e_1$, then one corresponding to $e_2$, then one corresponding to $e_2+e_1$, then one corresponding to $e_1$, then cross again one corresponding to $e_2+e_1$ and so on, crossing $2m-1$ hyperplanes in total. Finally, after we crossed all $4m$ hyperplanes, we reached the region associted to the full filtered chain. If $m$ is odd, the roles of $e_1$ and $e_2$ are interchanged. \qed

			\end{proof}

			The proof for $G_2 = I_2(6)$ is more involved. We again want to embed the positive regions of $\Shi^{(m-1)}(G_2)$ into the positive regions of $\Shi^{(m)}(G_2)$. To do this, we first have to generalize the notion. The positive roots are given by
			$$\{\alpha,\beta,\alpha+\beta,2\alpha+\beta,3\alpha+\beta,3\alpha+2\beta \},$$
			where one possible choice is $\alpha = e_2-e_1, \beta = 2e_2-e_2-e_1$, compare e.g. \cite[Section 2.10]{humphreys}. For $k \leq m$, define $\Shi^{(m,k)}(G_2)$ to be the collection of the hyperplanes in $\Shi^{(m)}(G_2)$ other than $H_{3\alpha+2\beta}^{(i)}$ for $i > k$, see the shaded part on the right in Fig.~\ref{arrangementG24} for an example. Furthermore, define $\Cat^{(m,k)}(G_2;q) := \sum_R q^{\coh(R)}$, where the sum ranges over all positive regions of $\Shi^{(m,k)}(G_2)$. In particular, $\Cat^{(m,m)}(G_2;q) = \Cat^{(m)}(G_2;q)$.

			\begin{proof}[for $k = 6$]
				\begin{figure}
					%\centering
					\begin{tabular}{c}
						
\setlength{\unitlength}{4.5pt}
\begin{picture}(66,65)(-4,-5)

	\color{sehrhellgrau}
	\thinlines
	\multiput(0.1,9.9)(0.05,-0.05){98}{\line(-3,1){5}}

	\color{sehrhellgrau}
	\thinlines
	\multiput(0.1,0.1)(0.05,  0){98}{\line(-3,2){5}}

	\color{sehrhellgrau}
	\thinlines
	\multiput(0.1,10.1)(0.05,  0){98}{\line(-3,2){5}}

%	\color{sehrhellgrau}
%	\thinlines
%	\multiput(0.1,19.9)(0.05,-0.05){98}{\line(-3,1){5}}

	\color{white}
	\thinlines
	\multiput(-0.1,0)(-0.05,0){100}{\line(0,1){25}}

	\color{grau}
	\thicklines
	
	\put(2.5,0.5){\line(1, 0){4.6}}
	\put(7  ,0.5){\line(0, 1){54.5}}
	\put(7  , 55){\line(1, 0){48}}

	\color{grau}
	\thicklines
	\multiput(10.1, 0)(0.125,0){400}{\line(0,1){50}}

	\color{black}
	\thinlines
	
	\put(0,-1){\line(0,1){61}}
	\multiput(10,0)(10,0){5}{\line(0,1){60}}

	\put(-1,-4){\scriptsize $H_{\alpha}^{(0)}$}

	\put(-1,0){\line(1,0){61}}
	\multiput(0,10)(0,10){5}{\line(1,0){60}}

	\put(-5, -0.5){\scriptsize $H_{\beta}^{(0)}$}
	
	\put(-0,10){\line(1,-1){10}}
	\put(-0,20){\line(1,-1){20}}
	\put(-0,30){\line(1,-1){30}}
	\put(-0,40){\line(1,-1){40}}
	\put(-0,50){\line(1,-1){50}}

	\put(10,0){\line(-2,1){10}}
	\put(20,0){\line(-2,1){20}}
	\put(30,0){\line(-2,1){30}}
	\put(40,0){\line(-2,1){40}}
	\put(50,0){\line(-2,1){50}}

	\put(10,0){\line(-3,1){10}}
	\put(20,0){\line(-3,1){20}}
	\put(30,0){\line(-3,1){30}}
	\put(40,0){\line(-3,1){40}}
	\put(50,0){\line(-3,1){50}}

	\color{red}
	\multiput( 4.8,0)(0.1,0){5}{\line(-3,2){ 5}}
	\multiput( 9.8,0)(0.1,0){5}{\line(-3,2){10}}
	\multiput(14.8,0)(0.1,0){5}{\line(-3,2){15}}
	\multiput(19.8,0)(0.1,0){5}{\line(-3,2){20}}
	\multiput(24.8,0)(0.1,0){5}{\line(-3,2){25}}

	\color{black}
	\put(-7,16.6){\scriptsize $H_{3\alpha+2\beta}^{(5)}$}

	\qbezier(57,58)(60,61)(63,58)
	\put(57.8,58.7){\vector(-1,-1){1.5}}
	
	\qbezier(1,1)(-4,2)(-6,5)
	\put(1,1){\vector(1,0){0.5}}
	
	\put(63,55){$R^\infty$}
	\put(-9,6){$R^0$}

	\put(.3,   7){\small I}
	\put(.3,10.5){\small II}
	\put(.3,  17){\small I}

\end{picture}
						
					\end{tabular}
					\caption{The embedding of the positive regions of $\Shi^{(4,3)}(G_2)$ into the positive regions of $\Shi^{(5,5)}(G_2)$; the remaining region walk from $R^\infty$ to $R^0$ and the \lq additional regions\rq.}
					\label{arrangementG24}
				\end{figure}
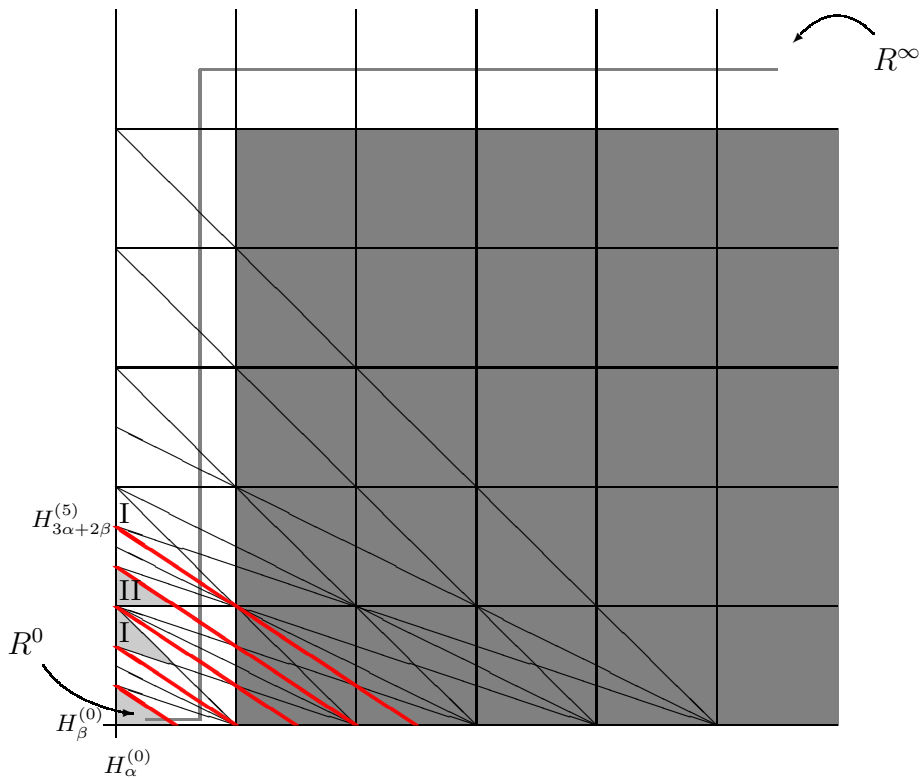

				For $k \geq 2$, it is again immediate how we can embed the positive regions of $\Shi^{(m-1,k-2)}(G_2)$ into the positive regions of $\Shi^{(m,k)}(G_2)$ as indicated in Fig.~\ref{arrangementG24}; and that this embedding increases the coheight by $1$. As in the proof for $k = 4$, we want to sum over all positive regions of $\Shi^{(m,k)}(G_2)$ which do not come from positive regions of $\Shi^{(m-1,k-2)}(G_2)$. We obtain as well a region walk from $R^\infty$ to $R^0$ with $\sum_R q^{\coh(R)} = [5m+k+1]_q$, where the sum ranges over all regions in the region walk. Now, there are two types of \lq additional regions\rq\ remaining which do not lie in the region walk either. These are confined by the hyperplane $H_\alpha^{(0)}$ and either hyperplanes of the form $H_{\alpha + \beta}^{(i)}, H_{3 \alpha + \beta}^{(3i-1)}$ for $1 \leq i \leq \lceil \frac{m-1}{3} \rceil$ or hyperplanes of the form $H_{\beta}^{(i)}, H_{3 \alpha + 2 \beta}^{(3i+3)}$ for $1 \leq i \leq \lfloor \frac{k-1}{3} \rfloor$. The former are labelled in the picture by I, the latter by II. In total, we obtain
				\begin{eqnarray}
					\Cat^{(m,k)}(G_2;q) &=& q \Cat^{(m-1,k-2)}(G_2;q) + [5 m + k + 1]_q \nonumber \\
														&& + \sum_{\ell = 1}^{\lfloor \frac{k}{3} \rfloor + \lfloor \frac{k-1}{3} \rfloor} q^{5m + k - 5 \ell} + \sum_{\ell = \lfloor \frac{k}{3} \rfloor + 1}^{\lceil \frac{m-1}{3} \rceil} q^{5m + 4 -7 \ell}. \label{eq:recurrence4}
				\end{eqnarray}
				Here, the $q$ in front of $\Cat^{(m-1,k-2)}(G_2;q)$ comes from the fact that the embedding increases the coheight by $1$; $[5 m + k + 1]_q$ is obtained from the region walk from $R^\infty$ to $R^0$; the first sum is obtained from the \lq additional regions\rq\ below the hyperplane $H_{3\alpha+2\beta}^{(k)}$; and the second sum is obtained from the \lq additional regions\rq\ above this hyperplane.

				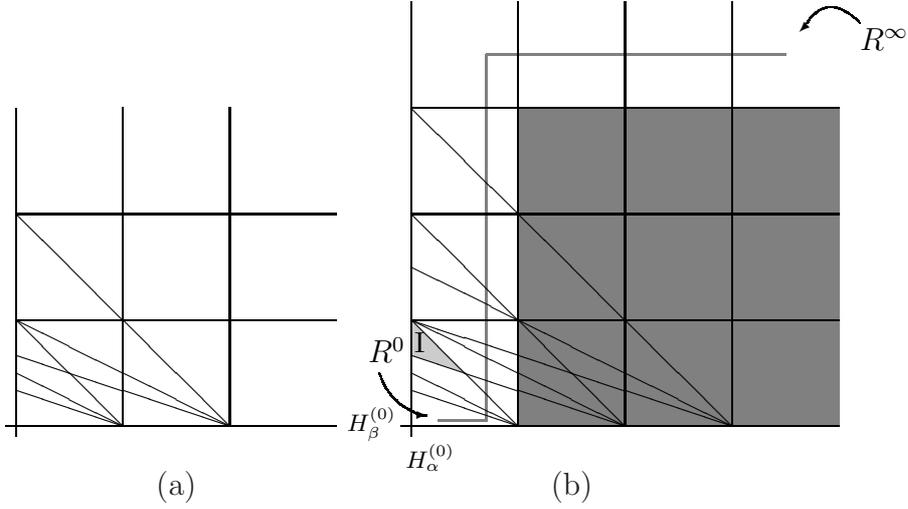
\begin{figure}
					%\centering
					\begin{tabular}{ccc}
						
\setlength{\unitlength}{4pt}
\begin{picture}(19,35)(6,-3)

	\color{black}
	\thinlines
	
	\put(0,-1){\line(0,1){31}}
	\multiput(10,0)(10,0){2}{\line(0,1){30}}

	\put(-1,0){\line(1,0){31}}
	\multiput(0,10)(0,10){2}{\line(1,0){30}}
	
	\put(10,0){\line(-1,1){10}}
	\put(20,0){\line(-1,1){20}}

	\put(10,0){\line(-2,1){10}}
	\put(20,0){\line(-2,1){20}}

	\put(10,0){\line(-3,1){10}}
	\put(20,0){\line(-3,1){20}}

\end{picture}
						
						& &

\setlength{\unitlength}{4pt}
\begin{picture}(43,45)(-6,-3)

	\color{grau}
	\thicklines
	
	\put(2.5,0.5){\line(1, 0){4.6}}
	\put(7  ,0.5){\line(0, 1){34.5}}
	\put(7  , 35){\line(1, 0){28}}

	\color{sehrhellgrau}
	\thinlines
	\multiput(4.9,5.1)(-0.05,0.05){98}{\line(-3,1){5}}

	\color{white}
	\thinlines
	\multiput(-0.1,6.5)(-0.05,0){98}{\line(0,1){5.2}}

	\color{grau}
	\thicklines
	\multiput(10.1, 0)(0.125,0){240}{\line(0,1){30}}

	\color{black}
	\thinlines
	
	\put(0,-1){\line(0,1){41}}
	\multiput(10,0)(10,0){3}{\line(0,1){40}}

	\put(-0.5,-4){\scriptsize $H_{\alpha}^{(0)}$}

	\put(-1,0){\line(1,0){41}}
	\multiput(0,10)(0,10){3}{\line(1,0){40}}

	\put(-6,-0.5){\scriptsize $H_{\beta}^{(0)}$}
	
	\put(10,0){\line(-1,1){10}}
	\put(20,0){\line(-1,1){20}}
	\put(30,0){\line(-1,1){30}}

	\put(10,0){\line(-2,1){10}}
	\put(20,0){\line(-2,1){20}}
	\put(30,0){\line(-2,1){30}}

	\put(10,0){\line(-3,1){10}}
	\put(20,0){\line(-3,1){20}}
	\put(30,0){\line(-3,1){30}}

	\qbezier(37,38)(39,41)(42,38)
	\put(37.8,38.7){\vector(-1,-1){1.5}}
	
	\qbezier(1,1)(-2,2)(-3,5)
	\put(1,1){\vector(1,0){0.5}}
	
	\put(42,35.5){$R^\infty$}
	\put(-4,6){$R^0$}

	\put(.3,   7){\small I}

\end{picture}
						
\\
						(a) && (b)
					\end{tabular}
					\caption{(a) The positive regions of $\Shi^{(2,0)}(G_2)$. (b) Their embedding into the positive regions of $\Shi^{(3,0)}(G_2)$; the remaining region walk from $R^\infty$ to $R^0$ and the \lq additional region\rq.}
					\label{arrangementG230}
				\end{figure}

				The next step is to embed the positive regions of $\Shi^{(m-1,0)}(G_2)$ into the positive regions of $\Shi^{(m,0)}(G_2)$ in the same sense as above, compare Fig.~\ref{arrangementG230}. By taking the region walk from $R^\infty$ to $R^0$ and the \lq additional regions\rq\ into account, we obtain
				\begin{eqnarray}
					\Cat^{(m,0)}(G_2;q) &=& q \Cat^{(m-1,0)}(G_2;q) + [5 m + 1]_q + \sum_{\ell = 1}^{\lceil \frac{m-1}{3} \rceil} q^{5m + 4 - 7 \ell}. \label{eq:recurrence2}
				\end{eqnarray}
				The $q$ in front of $\Cat^{(m-1,0)}(G_2;q)$ comes again from the fact that the embedding increases the coheight by $1$; $[5 m + 1]_q$ is obtained from the region walk from $R^\infty$ to $R^0$; and the sum is obtained from the \lq additional regions\rq\ confined by the hyperplane $H_\alpha^{(0)}$ and hyperplanes of the form $H_{\alpha + \beta}^{(i)}, H_{3 \alpha + \beta}^{(3i-1)}$ for $1 \leq i \leq \lceil \frac{m-1}{3} \rceil$. The additional region in the figure is labelled by I.
				
				Moreover, one has the obvious relations
				\begin{eqnarray}
					\Cat^{(0,0)}(G_2;q) &=& 1, \label{eq:recurrence1}\\
					\Cat^{(m,1)}(G_2;q) &=& \Cat^{(m,0)}(G_2;q) + q^{5m+1}. \label{eq:recurrence3}
				\end{eqnarray}
				The relations (\ref{eq:recurrence4})--(\ref{eq:recurrence3}) uniquely determine $\Cat^{(m,k)}(G_2;q)$. In particular, we get
				$$\Cat^{(1,1)}(G_2;q) \stackrel{(\ref{eq:recurrence3})}{=} q^6 + \Cat^{(1,0)}(G_2;q) \stackrel{(\ref{eq:recurrence2})}{=} q^6 + [6]_q + q \Cat^{(0,0)}(G_2;q) \stackrel{(\ref{eq:recurrence1})}{=} [7]_q + q,$$
				which again proves Corollary~\ref{qtFCNdihedralgroup} in this case. We want to obtain the recurrence relation
				\begin{eqnarray*}
					\Cat^{(m,m)}(G_2;q) &=& [6m+1]_q + q\Cat^{(m-1,m-1)}(G_2;q) \\
															&\stackrel{(\ref{eq:recurrence4})}{=}& q \Cat^{(m-1,m-2)}(G_2;q) + [6 m + 1]_q + \sum_{\ell = 1}^{\lceil \frac{m-1}{3} \rceil + \lfloor \frac{m-1}{3} \rfloor} q^{6m - 5 \ell}
				\end{eqnarray*}
				for $m \geq 2$. Thus, after a shift in $m$, it remains to show that the difference $\Cat^{(m,m)}(G_2;q) - \Cat^{(m,m-1)}(G_2;q)$ is given by
				%\begin{eqnarray*}
					$ \sum_{\ell = 1}^{\lceil \frac{m}{3} \rceil + \lfloor \frac{m}{3} \rfloor} q^{6m+5-5\ell}
																											= q^{6m}\ \big[\lceil \frac{m}{3} \rceil + \lfloor \frac{m}{3} \rfloor \big]_{q^{-5}}.$
				%\end{eqnarray*}

				\begin{figure}
					%\centering
					\begin{tabular}{c}
\setlength{\unitlength}{5pt}
\begin{picture}(65,30)(-6,-5)

	\color{grau}
	\thinlines
	\multiput(0.1,0.1)(0.05,  0){98}{\line(-3,2){5}}

	\color{sehrhellgrau}
	\thinlines
	\multiput(0.1,9.9)(0.05,-0.05){98}{\line(-3,1){5}}

	\color{sehrhellgrau}
	\thinlines
	\multiput(0.1,10.1)(0.05,  0){98}{\line(-3,2){5}}

	\color{white}
	\thinlines
	\multiput(-0.1,-0.1)(-0.05,0){100}{\line(0,1){25}}

	\color{grau}
	\thinlines
	\multiput(10.1,0.1)(0.05,  0){98}{\line(-3,2){5}}

	\color{sehrhellgrau}
	\thinlines
	\multiput(10.1,9.9)(0.05,-0.05){98}{\line(-3,1){5}}

	\color{white}
	\thinlines
	\multiput(9.9,-0.1)(-0.05,0){100}{\line(0,1){25}}

	\color{grau}
	\thinlines
	\multiput(20.1,0.1)(0.05,  0){98}{\line(-3,2){5}}

	\color{white}
	\thinlines
	\multiput(19.9,-0.1)(-0.05,0){100}{\line(0,1){25}}

	\color{black}
	\thinlines
	
	\put(0,-1){\line(0,1){25}}
	\multiput(10,0)(10,0){5}{\line(0,1){24}}

	\put(-0.5,-4){\scriptsize $H_{\alpha}^{(0)}$}

	\put(-1,0){\line(1,0){61}}
	\multiput(0,10)(0,10){2}{\line(1,0){60}}

	\put(-6, -0.5){\scriptsize $H_{\beta}^{(0)}$}
	
	\put(10,0){\line(-1,1){10}}
	\put(20,0){\line(-1,1){20}}
	\put(30,0){\line(-1,1){24}}
	\put(40,0){\line(-1,1){24}}
	\put(50,0){\line(-1,1){24}}

	\put(10,0){\line(-2,1){10}}
	\put(20,0){\line(-2,1){20}}
	\put(30,0){\line(-2,1){30}}
	\put(40,0){\line(-2,1){40}}
	\put(50,0){\line(-2,1){48}}

	\put(10,0){\line(-3,1){10}}
	\put(20,0){\line(-3,1){20}}
	\put(30,0){\line(-3,1){30}}
	\put(40,0){\line(-3,1){40}}
	\put(50,0){\line(-3,1){50}}

	\multiput( 4.9,0)(0.1,0){3}{\line(-3,2){ 5}}
	\multiput( 9.9,0)(0.1,0){3}{\line(-3,2){10}}
	\multiput(14.9,0)(0.1,0){3}{\line(-3,2){15}}
	\multiput(19.9,0)(0.1,0){3}{\line(-3,2){20}}

	\color{red}
	\multiput(24.8,0)(0.1,0){5}{\line(-3,2){25}}

	\color{black}
	\put(-6,16.6){\scriptsize $H_{3\alpha+2\beta}^{(5)}$}
	
	\put(  .5, .5){\small $1_1$}
	\put(10.5, .5){\small $2_2$}
	\put(20.5, .5){\small $3_3$}

	\put(  .5,  7){\small $2_1$}
	\put(10.5,  7){\small $3_2$}

	\put(  .5,10.5){\small $3_1$}

	\qbezier(-0.5,1)(-5,2)(-6,5)
	\put(-0.5,1){\vector(1,0){0.8}}
	\put(-7,6){$R^0$}

\end{picture}
					\end{tabular}
					\caption{Introducing the hyperplane $H_{3\alpha+2\beta}^{(5)}$ into $\Shi^{(5,4)}(G_2)$.}
					\label{arrangementG243}
				\end{figure}
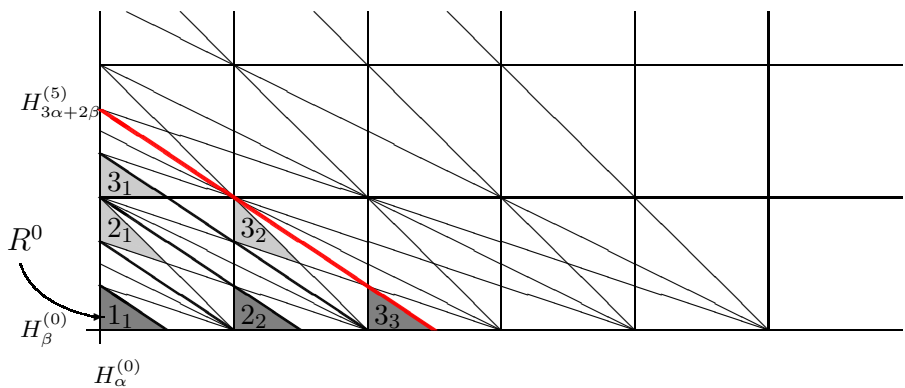

				From the discussion after (\ref{eq:recurrence4}), this is equivalent to show that
				\begin{eqnarray*}
					\Cat^{(m,m)}(G_2;q) - \Cat^{(m,m-1)}(G_2;q) = \sum_R q^{\coh(R)},
				\end{eqnarray*}
				where the sum ranges over the fundamental region $R^0$ and the \lq additional regions\rq\ below the hyperplane $H_{3\alpha+2\beta}^{(m)}$; those regions are indicated in Fig.~\ref{arrangementG24} in light grey.
				
				Obviously, $\Shi^{(m,m)}(G_2)$ is obtained from $\Shi^{(m,m-1)}(G_2)$ by introducing the hyperplane $H_{3\alpha+2\beta}^{(m)}$. Thus, 
				$$\Cat^{(m,m)}(G_2;q) - \Cat^{(m,m-1)}(G_2;q) = \sum_R q^{\coh(R)} - \sum_{R'} q^{\coh(R')-1},$$
				where the first sum ranges over all \lq shifted additional regions\rq, labelled in Fig.~\ref{arrangementG243} by $i_j$ for $1 \leq j \leq i \leq \lceil \frac{m}{3} \rceil + \lfloor \frac{m}{3} \rfloor$; the second sum ranges over all \lq shifted additional regions\rq, labelled in the figure by $i_j$ for $j < i$; and where the coheight is considered in $\Shi^{(m,m)}(G_2)$. Here, shifted means that the confining hyperplane $H_\alpha^{(0)}$ is replaced by $H_\alpha^{(j-1)}$ for $1 \leq j \leq \lceil \frac{m}{3} \rceil + \lfloor \frac{m}{3} \rfloor$. The term in the second sum labelled by $i_j$ for $j < i$ cancels the term in the first labelled by $i_{j+1}$. In total, we obtain $\sum_R q^{\coh(R)}$, where the sum ranges over all \lq additional regions\rq\ labelled by $i_1$ for $1 \leq j \leq i \leq \lceil \frac{m}{3} \rceil + \lfloor \frac{m}{3} \rfloor$. Those regions are exactly the fundamental region $R^0$ and the \lq additional regions\rq\ below the hyperplane $H_{3\alpha+2\beta}^{(m)}$. This completes the proof. \qed
			\end{proof}
			
			\begin{remark}
				It follows from (\ref{eq:recurrence2}) that we cannot generalize Corollary~\ref{qtFCNdihedralgroup} to $m \geq 2$ for non-crystallographic dihedral groups in terms of the extended Shi arrangement or, equivalently, in terms of filtered chains in D.~Armstrong's suggested \lq root poset\rq.
			\end{remark}

				\subsection{The cyclic groups}
				
				We can simply compute the $q,t$-Fu\ss-Catalan numbers for the cyclic groups: the described action of $\mathcal{C}_k = G(k,1,1)$ on $\C[x,y]$ is given by $\zeta_k(x^ay^b) = \zeta_k^{a-b} x^ay^b$, where $\zeta_k$ is a $k$-th root of unity. This gives
				\begin{eqnarray*}
					\C[x,y]^{C_k} &=& \span \big\{ x^ay^b : a \equiv b \mod k \big\}, \\
					\C[x,y]^{\det} &=& \span \big\{ x^ay^b : a \equiv b+1 \mod k \big\},
				\end{eqnarray*}
				and therefore, $\C[x,y]^{\det} = x \C[x,y]^W + y^{k-1} \C[x,y]^W$. Thus, $\Cat^{(1)}(\mathcal{C}_k;q,t) = q + t^{k-1}$, and more generally,
				\begin{eqnarray*}
					\Cat^{(m)}(\mathcal{C}_k;q,t) &=& \sum_{i=0}^m q^{i}t^{(m-i)(k-1)}.
				\end{eqnarray*}

				\begin{corollary} \label{co: qtfusscat cyclic groups}
					Conjectures~\ref{conjecturedim} and \ref{conjectureqtdiag} hold for the cyclic groups: let $\mathcal{C}_k = G(k,1,1)$ be the cyclic group of order $k$ acting diagonally on $\C[x,y]$ as described above. Let $N = k-1$ be the number of reflections in $\mathcal{C}_k$ and let $N^* = 1$ be the number of reflecting hyperplanes. Then
					\begin{eqnarray*}
						q^{mN} \Cat^{(m)}(C_k;q,q^{-1}) &=& q^{mN^*} \Cat^{(m)}(C_k;q^{-1},q) \\
																						&=& 1 + q^k + \dots + q^{mk}.
					\end{eqnarray*}
				\end{corollary}
				As the cyclic group $\mathcal{C}_k$ is \emph{not} crystallographic for $k \geq 3$, Conjecture~\ref{conjectureqtq} is not relevant in this case.
	
	\section{Connections to rational Cherednik algebras}\label{rationalcherednikalgebras}
	
		In this section, we investigate (conjectured) connections between the generalized diagonal coinvariants $\DR^{(m)}(W)$ and a module that naturally arises in the context of rational Cherednik algebras. These algebras were introduced by P.~Etingof and V.~Ginzburg in \cite{EG} and where further studied by Y.~Berest, P.~Etingof and V.~Ginzburg in \cite{BEG2,BEG}. Most of the facts about rational Cherednik algebras are taken from these references.
	
		The work by Y.~Berest, P.~Etingof and V.~Ginzburg deals only with real reflection groups but in \cite{griffeth}, S.~Griffeth partially generalized the work to complex reflection groups. We expect that the conjectured connection of the $q,t$-Fu\ss-Catalan numbers can also be transferred to this generalized context.
		
		In this section, fix $W$ to be a real reflection group acting on a complex vector space $V$. Note that $V = V' \otimes_\R \C$ is the complexification of the real vector space $V'$ on which $W$ naturally acts. Furthermore let $T \subseteq W$ be the set of reflections in $W$ and let $W$ act on $T$ by conjugation (note in particular that $\omega t \omega^{-1} \in T$).
		
		\subsection{The rational Cherednik algebra}
	
			P.~Etingof and V.~Ginzburg defined the rational Cherednik algebra as follows \cite{EG}:
			\begin{definition}
				Let $c : T \rightarrow \C, t \mapsto c_t$ be a $W$-invariant function on the set of reflections. The \Dfn{rational Cherednik algebra} $\Hc = \Hc(W)$ is the associative algebra generated by the vector spaces $V, V^*$ and the reflection group $W$, subject to the defining relations
				\begin{eqnarray}
					\omega x - \omega(x) \omega = \omega y - \omega(y) \omega = 0 &\mathrm{\ for\ all\ }& y \in \h, x \in \h^*, \omega \in W, \nonumber \\
					{[}x_1, x_2] = [y_1, y_2] = 0 & \mathrm{\ for\ all\ } & y_1,y_2 \in \h, x_1,x_2 \in \h^*, \nonumber \\
					{[}y, x] = \langle y,x \rangle - \sum_{t \in T}c_t \langle y,\alpha_t \rangle \langle \alpha_t^{\vee},x \rangle t & \mathrm{\ for\ all\ } & y \in \h,x \in \h^*. \label{eq:inhomrelation}
				\end{eqnarray}
				Here, $[a,b] := ab - ba$ denotes the \Dfn{commutator} of $a$ and $b$, $\langle \cdot,\cdot \rangle$ denotes the pairing on $V \times V^*$ and $\alpha_t \in V^*$ (resp. $\alpha_t^{\vee} \in V$) denotes the positive root (resp. positive coroot) associated to $t \in T$.
			\end{definition}

			The polynomial ring $\C[\h]$ sits inside $\Hc$ as the subalgebra generated by $\h^*$ and the polynomial ring $\C[\h^*]$ as the subalgebra generated by $\h$. Furthermore, the elements in $W$ span a copy of the group algebra $\C[W]$ sitting naturally inside $\Hc$.

			The \Dfn{spherical subalgebra} of $\Hc$ is defined as $\e \Hc \e \subseteq \Hc$ where $\e$ is the trivial idempotent viewed as an element in $\Hc$.
		
		\subsubsection{An induced grading}
	
			Define the \Dfn{degree operator} $\mathbf{h} \in \Hc$ by $\mathbf{h} := \frac{1}{2} \sum_i(x_i y_i + y_i x_i)$, where $\{x_i\}$ and $\{y_i\}$ are dual bases of $\h^*$ and of $\h$ respectively. In \cite[Eq.~(2.6)]{BEG2}, it is shown how the generators of $\Hc$ behave under the \Dfn{adjoint action} $\adh : u \mapsto [\mathbf{h},u]$,
			\begin{eqnarray}
				\adh (x) = x, \quad \adh (y) = -y , \quad \adh (\omega) = 0 \quad \mathrm{for} \quad x \in \h^*, y \in \h, \omega \in W. \label{eq:degreeoperator}
			\end{eqnarray}
			As the defining relations for $\Hc$ become homogeneous, this gives an induced grading
			$$\Hc = \bigoplus_{\lambda \in \Z} \{a \in \Hc : [\mathbf{h},a] = \lambda a\}.$$
			We denote the degree with respect to the $\adh$-grading by $\hdeg$. From (\ref{eq:degreeoperator}) it follows immediately that
			\begin{eqnarray}
				\hdeg x = 1, \quad \hdeg y = -1, \quad \hdeg \omega = 0 \quad \mathrm{for} \quad x \in \h^*, y \in \h, \omega \in W. \label{eq:degdiag}
			\end{eqnarray}
	
		\subsubsection{A natural filtration} \label{sec:filtration}
	
			Assign a \Dfn{total degree} on $\Hc$ by
			$$\totdeg x = \totdeg y = 1, \quad \totdeg \omega = 0 \quad \mathrm{for} \quad x \in \h^*, y \in \h, \omega \in W.$$
			For $c = c_t \not\equiv 0$, the defining relation (\ref{eq:inhomrelation}) now become inhomogeneous. Thus, one only gets an (increasing) filtration $\Fdot(\Hc)$, where $F_k(\Hc)$ denotes all elements in $\Hc$ of total degree less than or equal to $k$ (in particular, $F_k(\Hc) = \{0\}$ for negative $k$). This filtration is stable under the adjoint action; hence, it induces an $\adh$-action on the associated graded algebra
			$$\gr(\Hc) := \bigoplus_{k \geq 0} F_k(\Hc)/F_{k-1}(\Hc).$$
			Observe that the commutator $[y,x]$ is annihilated in $\gr(\Hc)$, i.e., $yx - xy = 0 \in \gr(\Hc)$. By \cite[Theorem 1.3]{EG}, one has a algebra isomorphism
			\begin{eqnarray}
				\gr(\Hc) &\cong& \C[\h \oplus \h^*] \rtimes W, \label{eq:grHcisomorphism}
			\end{eqnarray}
			which is graded with respect to the grading induced by $\adh$ and also with respect to the grading induced by the filtration. As the degree operator commutes with the trivial idempotent $\e \in \Hc$ as well as with the determinantal idempotent $\eeps \in \Hc$, the filtration on $\Hc$ carries over to the spherical subalgebra $\e \Hc \e$ and to $\e \Hc \eeps$, and one gets graded isomorphisms
			\begin{eqnarray}
				\gr(\e \Hc \e) \cong \e \C[\h \oplus \h^*], \quad \gr(\e \Hc \eeps) \cong \eeps \C[\h \oplus \h^*], \label{eq:greHceisomorphism}
			\end{eqnarray}
			see (7.8) in \cite{BEG}.

		\subsection{A module over the rational Cherednik algebra}
	
			For any $\Hc$-representation $\tau$, define a $\Hc$-module $\mathcal{M}(\tau)$ to be the induced module
			$$\mathcal{M}_c(\tau) := \Hc \otimes_{\C[\h] \rtimes W} \tau,$$
			where $\C[\h] \rtimes W$ acts on $\tau$ by $p \omega \cdot a := p(0)(\omega(a))$ for $p \in \C[\h], \omega \in W$ and $a \in \tau$.
	
			For our purposes it is enough to restrict to the case where $\tau$ is the trivial representation and where the parameter $c = c_t = \frac{1}{h} + m$ is a rational constant (as above, $h$ denotes the Coxeter number associated to $W$). Set $\Hm{m} := \Hc$ and denote this $\Hm{m}$-module by $\Mm{m} := \mathcal{M}_c(\C)$. Many properties even hold in the more general context of $\tau$ being any $\Hc$-representation and for arbitrary $W$-invariant parameter $c$.
	
			C.F.~Dunkl and E.~Opdam showed in \cite{dunklopdam} that $\Mm{m}$ has a unique simple quotient, which we denote by $\Lm{m} = L_c(\C)$ for $c = \frac{1}{h} + m$. In \cite{BEG}, Y.~Berest, P.~Etingof and V.~Ginzburg investigated this $\Hm{m}$-module and showed in \cite[Theorem 1.4]{BEG} that $\Lm{m}$ is the only finite dimensional irreducible $\Hm{m}$-module. Moreover, they computed the graded Hilbert series of $\Lm{m}$ with respect to the grading induced by the adjoint action, see \cite[Theorem 1.6]{BEG}.
			\begin{theorem}[Y.~Berest, P.~Etingof, V.~Ginzburg]\label{theoBEG1}
				Let $\Lm{m}$ be the unique simple $\Hm{m}$-module. The Hilbert series of $\Lm{m}$ with respect to the grading induced by $\adh$ is given by
				$$\mathcal{H}(\Lm{m};q) = q^{-mN}([mh+1]_q)^\ell,$$
				where $\ell$ is the rank of $W$ and $N$ is the number of positive roots. In particular,
				$$\dim \Lm{m} = (mh+1)^\ell.$$
			\end{theorem}

			Observe that $\e \Lm{m} \subseteq \Lm{m}$ has a natural $\e \Hm{m} \e$-module structure and that the degree operator $\mathbf{h}$ preserves $\e \Lm{m}$. Hence, $\adh$ also induces a grading on $\e \Lm{m}$. For the following theorem see \cite[Theorem 1.10]{BEG}:
			\begin{theorem}[Y.~Berest, P.~Etingof, V.~Ginzburg]\label{theoBEG}
				$\e \Lm{m}$ is the only finite dimensional simple $\e \Hm{m} \e$-module. The Hilbert series of $\e \Lm{m}$ with respect to the grading induced by $\adh$ is given by
				$$\mathcal{H}(\e \Lm{m};q) = q^{-mN} \prod_{i=1}^\ell{\frac{[d_i+mh]_q}{[d_i]_q}},$$
				where $\ell$ is the rank of $W$, $d_1 \leq \dots \leq d_\ell$ are the degrees and $h$ is the Coxeter number.
			\end{theorem}

%			Y.~Berest, P.~Etingof and V.~Ginzburg proved Theorem~\ref{theoBEG1} and Theorem~\ref{theoBEG} by constructing a certain homogeneous system of parameters (hsop) inside the $p$-th component of $\C[V] = \operatorname{Sym}(V^*)$ such that its span is $W$-stable and carries a $W$-representation isomorphic to $V^*$. In \cite[Section 4]{bessisreiner}, D.~Bessis and V.~Reiner described how such a construction would imply both theorems for arbitrary complex reflection groups and conjectured in \cite[Conjecture 4.3]{bessisreiner} that such hsop exist at least for well-generated groups. Furthermore, they constructed a hsop for any group $G(k,p,\ell)$. This implies Theorems~\ref{theoBEG1} and \ref{theoBEG} in this more general context.
	
		\subsubsection{Filtrations on $\Lm{m}$ and on $\e \Lm{m}$ and the induced graded modules}
			
			There exist nice and important decompositions of the modules in question. These are based on an algebra isomorphism
			$$\e \Hm{m} \e \hspace{5pt} \tilde{\longrightarrow} \hspace{5pt} \eeps \Hm{m+1} \eeps,$$
			which was discovered in \cite[Proposition 4.6]{BEG} and in \cite[Proposition 4.8]{gordon}. Y.~Berest, P.~Etingof and V.~Ginzburg \cite[Lemma 4.7, Proposition 4.8]{BEG} used this isomorphism to connect the modules $\Lm{m+1}$ and $\Lm{m}$ by the isomorphism
			\begin{eqnarray}
				\Lm{m+1} &\cong& \Hm{m+1} \eeps \otimes_{\e \Hm{m} \e} \e \Lm{m}. \label{eq:LmLm1}
			\end{eqnarray}
 			For $m = 1$, this isomorphism was studied by I.~Gordon in \cite[Theorem 4.9]{gordon}. In \cite[Lemma 4.6]{gordon}, he moreover shows that $\Lm{0} \cong \C$ carries the trivial representation. Applying (\ref{eq:LmLm1}) iteratively gives rise to a decomposition of the $\Hm{m}$-module $\Lm{m}$, which is taken from \cite[Eq.~(7.6)]{BEG} as
			\begin{eqnarray}
				\hspace*{-5pt} \Lm{m} &\cong& \Hm{m} \eeps \otimes_{\e \Hm{m-1} \e} \e \Hm{m-1} \eeps \otimes_{\e \Hm{m-2} \e} \dots \otimes_{\e \Hm{1} \e} \e \Hm{1} \eeps \otimes_{\e \Hm{0} \e} \C, \label{decomposition1}
			\end{eqnarray}
			and also a decomposition of the $\e \Hm{m} \e$-module $\e \Lm{m}$, see \cite[Eq.~(7.7)]{BEG}, as
			\begin{eqnarray*}
				\e \Lm{m} &\cong& \e \Hm{m} \eeps \otimes_{\e \Hm{m-1} \e} \e \Hm{m-1} \eeps \otimes_{\e \Hm{m-2} \e} \dots \otimes_{\e \Hm{1} \e} \e \Hm{1} \eeps \otimes_{\e \Hm{0} \e} \C.
			\end{eqnarray*}
			
			From those decompositions one can see that the filtration on $\Hc$ described in Section~\ref{sec:filtration} carries over to $\Lm{m}$ and to $\e \Lm{m}$ via the tensor product filtration: let $R$ be a filtered $\C$-algebra, $A$ a filtered right $R$-module and $B$ a filtered left $R$-module. The \Dfn{tensor product filtration} $\Fdot(A \otimes B)$ is then defined by
			$$F_k(A \otimes_R B) = \sum_j F_j(A) \otimes_R F_{k-j}(B).$$
			With $\Lm{0} \cong \C$ sitting in degree $0$, the iterative application of the tensor product filtration to the decompositions of $\Lm{m}$ and of $\e \Lm{m}$ defines a filtration on $\Lm{m}$ and a filtration on $\e \Lm{m}$. We denote the associated graded modules by $\gr(\Lm{m})$ and by $\gr(\e \Lm{m})$ respectively.

			We have already seen that the degree operator $\mathbf{h}$ acts on $\Lm{m}$ and on $\e \Lm{m}$. The filtrations on $\Lm{m}$ and on $\e \Lm{m}$ are stable under the adjoint action; thus, $\adh$ induces a grading on $\gr(\Lm{m})$ and on $\gr(\e \Lm{m})$. Moreover, with respect to this grading, (\ref{eq:degreeoperator}) and (\ref{eq:degdiag}) still hold (compare \cite[Proof of Theorem 5]{gordon}); thus,
			\begin{eqnarray*}
				\mathcal{H}(\gr(\Lm{m});q) = \mathcal{H}(\Lm{m};q), \quad \mathcal{H}(\gr(\e \Lm{m});q) = \mathcal{H}(\e \Lm{m};q).
			\end{eqnarray*}

		\subsubsection{A bigrading on the associated graded modules}\label{bigrading}
		
			The two different gradings described above allows one to define bigradings on the various objects as follows, see \cite[Section 7.2]{BEG}:
			$$\gr(\Hc) = \bigoplus_{p,q \geq 0} (\gr(\Hc))_{p,q},$$
			where $(\gr(\Hc))_{p,q}$ consists of all elements in $\gr(\Hc)$ that are homogeneous with respect to both gradings and which have total degree $p+q$ and $\adh$-degree $p-q$. Observe that the isomorphisms (\ref{eq:grHcisomorphism}) and (\ref{eq:greHceisomorphism}) become bigraded and get one gets
			\begin{eqnarray*}
				\deg(x) = (1,0), \quad \deg(y) = (0,1), \quad \deg(\omega) = (0,0) \quad \mathrm{for} \quad x \in \h^*, y \in \h, \omega \in W,
			\end{eqnarray*}
			where $\deg$ denote the bidegree with respect to this bigrading.
			
			Similar considerations apply to the $\Hm{m}$-module $\Lm{m}$ and to the $\e \Hm{m} \e$-module $\e \Lm{m}$ turning $\gr(\Lm{m})$ and $\gr(\e \Lm{m})$ into bigraded modules.

		\subsubsection{$\gr(\Lm{m})$ and the generalized diagonal coinvariants} \label{sec: associated graded}
	
			As indicated in Theorem~\ref{theogordon}, I.~Gordon connected the diagonal coinvariant ring with the $\mathsf{H}^{(1)}$-module $\Lm{1}$. Using the above decompositions, we prove his theorem in a slightly more general context using the same argument. The following well-known lemma is taken from \cite[Lemma 6.7 (2)]{gordonstafford}:
			\begin{lemma}\label{tensorproductfiltrationsurjection}
				Let $R$ be a filtered $\C$-algebra, $A$ a filtered right $R$-module and $B$ a filtered left $R$-module. Then there is a natural surjection
				$$\gr A \otimes_{\gr(R)} \gr B \twoheadrightarrow \gr(A \otimes_R B).$$
			\end{lemma}

			\begin{theorem}\label{theogordon_new2}
				Let $W$ be a real reflection group, let $\DR^{(m)}(W)$ be the space of generalized diagonal coinvariants and let $\gr(\Lm{m})$ be bigraded as described in Section~\ref{bigrading}. Then there exists a natural surjection of bigraded $W$-modules,
				$$\DR^{(m)}(W) \otimes \det \twoheadrightarrow \gr(\Lm{m}).$$
			\end{theorem}

			\begin{proof}
				Consider the decomposition (\ref{decomposition1}),
				$$\Lm{m} \cong \Hm{m} \eeps \otimes_{\e \mathsf{H}^{(m-1)} \e} \e \mathsf{H}^{(m-1)} \eeps \otimes_{\e \mathsf{H}^{(m-2)} \e} \dots \otimes_{\e \mathsf{H}^{(1)} \e} \e \mathsf{H}^{(1)} \eeps \otimes_{\e \mathsf{H}^{(0)} \e} \C.$$
				By Lemma~\ref{tensorproductfiltrationsurjection}, (\ref{eq:grHcisomorphism}) and (\ref{eq:greHceisomorphism}), we get a surjection
				\begin{eqnarray*}
					(S \otimes \det) \otimes_{\e S} \eeps S \otimes_{\e S} \dots \otimes_{\e S} \eeps S \otimes_{\e S} \C \twoheadrightarrow \gr(\Lm{m}),
				\end{eqnarray*}
				where we write $S$ for $\C[\h \oplus \h^*]$. As the left-hand side equals $\DR^{(m)}(W) \otimes \det$, the theorem follows. \qed
			\end{proof}

%			\begin{remark} \label{re: griffeth}
%				Work of Griffeth \cite{griffeth} suggests that Theorem~\ref{theogordon_new2} holds also for the complex reflection group $G(k,p,l)$.
%			\end{remark}

			Theorem~\ref{theogordon_new2} partially generalizes Theorem~\ref{theogordon}:
			\begin{corollary}\label{theogordon_new}
				Let $W$ be a real reflection group. There exists a graded $W$-stable quotient ring $R_W$ of $\DR^{(m)}(W)$ such that
				\begin{itemize}
					\item[(i)] $\dim(R_W) = (mh+1)^\ell$ and
					\item[(ii)] $q^{mN} \mathcal{H}(R_W;q) = [mh+1]_q^\ell$.
				\end{itemize}
			\end{corollary}

			For real reflection groups, Theorems~\ref{theoBEG} and \ref{theogordon_new2} show that Conjectures~\ref{conjecturedim} and \ref{conjectureqtdiag} would be implied by the following conjecture, which is, for $m = 1$, due to M.~Haiman \cite[Conjecture 7.2.5]{haiman8}:
			\begin{conjecture}[M.~Haiman]\label{surjectiontrivialrepresentation}
				The kernel of the surjection defined in Theorem~\ref{theogordon_new2} does not contain a copy of the trivial representation.
			\end{conjecture}

			This conjecture would show that the module $M^{(m)}(W)$, and thereby the $q,t$-Fu\ss-Catalan numbers $\Cat^{(m)}(W;q,t)$, can be described in terms of the $\Hm{m}$-module $\Lm{m}$:
			\begin{corollary}\label{conjecture3corollary}
				Let $W$ be a real reflection group and let $M^{(m)}(W)$ be the bigraded $W$-module defined in Section~\ref{qtfusscatdefinition}. If Conjecture~\ref{surjectiontrivialrepresentation} holds, then
				$$M^{(m)}(W) \otimes \det \cong \mathbf{e} \left(\gr(L) \right).$$
				In particular, Conjecture~\ref{surjectiontrivialrepresentation} implies Conjectures~\ref{conjecturedim} and \ref{conjectureqtdiag} for real reflection groups.
			\end{corollary}

\begin{appendix}
\section{The computations with {\tt Singular} and {\tt Macaulay2}} \label{appendix}
	To compute the dimensions of the module $M^{(m)}(W)$ defined in Section \ref{qtfusscatdefinition} for the classical reflection groups, we used the computer algebra system {\tt Singular} \cite{singular}.
	\vspace*{10pt}

	\begin{center}
	{\bf Table 1}  $\dim M^{(m)}(W)$ for types $B$ and $D$:
	\vspace*{5pt}
	
	\begin{tabular}[ht]{c}
	 $m$  \vspace{4pt}   \\
	 $n=1$ \vspace{0.5pt} \\
	 $n=2$ \vspace{0.5pt} \\
	 $n=3$ \vspace{0.5pt} \\
	 $n=4$ \vspace{0.5pt}
	\end{tabular}
	\hspace{3pt}
	\begin{tabular}[ht]{c|c|c|c}
	  1 & 2 & 3 & 4 \\
	 \hline
	 \hline
	  2 & 3 & 4 & 5 \\
	 \hline
	  6 & 15 & 28 & 45 \\
	 \hline
	  20 & 84 &  & \\
	 \hline
	  70 & 495 &  &
	\end{tabular}
	\hspace{10pt}
	\begin{tabular}[ht]{c|c|c|c}
	   1 & 2 & 3 & 4 \\
	 \hline
	 \hline
	   1 & 1 & 1 & 1 \\
	 \hline
	   4 & 9 & 16 & 25 \\
	 \hline
	   14 & 55 & 140 & 285\\
	 \hline
	   50 & 336 &  &
	\end{tabular}
	\end{center}
	\vspace*{5pt}
	
	For the computations of the bigraded Hilbert series of $M^{(m)}$ we used the computer algebra system {\tt Macaulay 2} \cite{macaulay2}. We write $[n]$ for $[n]_{q,t}$.
	\vspace*{10pt}

	\begin{center}
	{\bf Table 2}  $\Cat^{(m)}(B_n,q,t)$:
	\vspace*{5pt}

	\begin{tabular}[ht]{r||l}
		$n=2,m=1$ & $[5]+qt[1]$ \\
	\hline
		$m=2$     & $[9]+qt[5]+q^2t^2[1]$ \\
	\hline
		$m=3$     & $[13]+qt[9]+q^2t^2[5]+q^3t^3[1]$ \\
	\hline
	\hline
		$n=3,m=1$ & $[10]+qt[6]+qt[4]$ \\
	\hline
		$m=2$			& $[19]+qt[15]+qt[13]+q^2t^2[11]+q^2t^2[9]+q^3t^3[7]+q^2t^2[7]+q^4t^4[3]$ \\
	\hline
		$m=3$			& $[28]+qt[24]+qt[22]+q^2t^2[20]+q^2t^2[18]+q^3t^3[16]+$ \\
							& $q^2t^2[16]+q^3t^3[14]+q^4t^4[12]+q^3t^3[12]+q^4t^4[10]+$ \\
							& $q^5t^5[8]+q^3t^3[10]+q^5t^5[6]+q^6t^6[4]$ \\
	\hline
	\hline
		$n=4,m=1$ & $[17]+qt[13]+qt[11]+q^2t^2[9]+qt[9]+q^3t^3[5]+q^2t^2[5]+q^4t^4[1]$ \\
	\hline
		$m=2$			& $[33]+ qt[29]+qt[27]+q^2t^2[25]+qt[25]+q^2t^2[23]+q^3t^3[21]+$ \\
							& $2q^2t^2[21]+q^3t^2[19]+q^4t^4[17]+q^2t^2[19]+2q^3t^3[17]+q^4t^4[15]+$ \\
							& $q^5t^5[13]+q^2t^2[17]+q^3t^3[15]+2q^4t^4[13]+q^5t^5[11]+q^6t^6[9]+$ \\
							& $q^3t^3[13]+q^4t^4[11]+2q^5t^5[9]+q^6t^6[7]+q^7t^7[5]+q^9t^9[1]+$ \\
							& $q^4t^4[9]+2q^6t^6[5]+q^8t^8[1]$
	\end{tabular}
	\end{center}
	\vspace*{10pt}

	\begin{center}
	{\bf Table 3}  $\Cat^{(m)}(D_n,q,t)$:
	\vspace*{5pt}

	\begin{tabular}[ht]{r||l}
		$n=2,m=1$ & $[3]+qt[1]$ \\
	\hline
		$m=2$     & $[5]+qt[3]+q^2t^2[1]$ \\
	\hline
		$m=3$     & $[7]+qt[5]+q^2t^2[3]+q^3t^3[1]$ \\
	\hline
	\hline
		$n=3,m=1$ & $[7]+qt[4]+qt[3]$ \\
	\hline
		$m=2$     & $[13]+qt[10]+qt[9]+q^2t^2[7]+q^2t^2[6]+q^2t^2[5]+q^3t^3[4]+q^4t^4[1]$ \\
	\hline
		$m=3$     & $[19]+qt[16]+qt[15]+q^2t^2[13]+q^2t^2[12]+q^3t^3[10]+q^2t^2[11]+$ \\
							& $q^3t^3[9]+q^4t^4[7]+q^3t^3[8]+q^4t^4[6]+q^5t^5[4]+q^3t^3[7]+q^5t^5[3]$ \\
	\hline
	\hline
		$n=4,m=1$ & $[13]+2qt[9]+qt[7]+ 2q^2t^2[5]+q^4t^4[1]+q^3t^3[1]$ \\
	\hline
		$m=2$			& $[25]+2qt[21]+qt[19]+3q^2t^2[17]+2q^2t^2[15]+$ \\
							& $4q^3t^3[13]+q^2t^2[13]+2q^3t^3[11]+5q^4t^4[9]+$ \\
							& $q^5t^5[7]+2q^6t^6[5]+q^8t^8[1]+q^4t^4[7]+$ \\
							& $2q^5t^5[5]+q^7t^7[1]+q^6t^6[1]$
	\end{tabular}
	\end{center}
	
	\vspace*{10pt}

	\begin{center}
	{\bf Table 4}  $\Cat^{(1)}(W,q,t)$ for several complex reflection groups of rank $2$:
	\vspace*{5pt}

		\begin{tabular}[ht]{r||l}\label{table:qtcatcomplexreflectiongroups}
			$G(3,1,2)$  &  $q^7 + q^5t + q^3t^2 + q^2t + qt^3 + t^5$\\
			\hline
			$G(4,1,2)$  &  $q^{10} + q^7t + q^4t^2 + q^3t + qt^3 + t^6$\\
			\hline
			$G(6,1,2)$  &  $q^{16} + q^{11}t + q^6t^2 + q^5t + qt^3 + t^8$\\
			\hline
			$G(3,2,2)$  &  $q^7 + q^5t + q^3t^2 + q^2t + qt^3 + t^5$\\
			\hline
			$G(4,2,2)$  &  $q^6 + q^4t^2 + 2q^3t + q^2t^4 + 2qt^3 + t^6$\\
			\hline
			$G(6,2,2)$  &  $q^{10} + q^6t^2 + 2q^5t + q^3t^5 + q^2t^4 + qt^3 + t^8$\\
		\end{tabular}
		\end{center}
\end{appendix}

\providecommand{\bysame}{\leavevmode\hbox to3em{\hrulefill}\thinspace}
\providecommand{\MR}{\relax\ifhmode\unskip\space\fi MR }
\providecommand{\MRhref}[2]{%
  \href{http://www.ams.org/mathscinet-getitem?mr=#1}{#2}
}
\providecommand{\href}[2]{#2}


\begin{thebibliography}{10}

\bibitem{alfano}
J.~Alfano, \emph{A basis for the ${Y}^2$ subspace of diagonal harmonics}, PhD
  thesis, University of California, San Diego, USA (1994).

\bibitem{armstrong2}
D.~Armstrong, \emph{Braid groups, clusters and free probability: an outline
  from the {AIM} workshop}, available at {\tt \scriptsize
  http://www.math.cornell.edu/\~{}armstron/} (2004).

\bibitem{armstrong1}
\bysame, \emph{Generalized noncrossing partitions and combinatorics of
  {C}oxeter groups}, PhD thesis, Cornell University, USA, to appear in Mem.\
  Amer.\ Math.\ Soc. (2006).

\bibitem{athanasiadis1}
C.A. Athanasiadis, \emph{Deformations of {C}oxeter hyperplane arrangements and
  their characteristic polynomials}, In Arrangements - Tokyo 1998, Adv.\ Stud.\
  Pure Math. \textbf{27} (2000), 1--26.

\bibitem{athanasiadis2}
\bysame, \emph{Generalized {C}atalan numbers, {W}eyl groups and arrangements of
  hyperplanes}, Bull.\ Lond.\ Math.\ Soc. \textbf{36} (2004), 294--302.

\bibitem{athanasiadis3}
\bysame, \emph{On a refinement of the generalized {C}atalan numbers for {W}eyl
  groups}, Trans.\ Amer.\ Math.\ Soc. \textbf{357} (2005), 179--196.

\bibitem{BEG2}
Y.~Berest, P.~Etingof, and V.~Ginzburg, \emph{Cherednik algebras and
  differential operators on quasi-invariants}, Duke Math.\ J. \textbf{118}
  (2003), no.~2, 279--337.

\bibitem{BEG}
\bysame, \emph{Finite dimensional representations of rational {C}herednik
  algebras}, Int.\ Math.\ Res.\ Not. \textbf{19} (2003), 1053--1088.

\bibitem{bessis4}
D.~Bessis, \emph{Finite complex reflection arrangements are $k(\pi,1)$},
  preprint, available at {\tt \scriptsize arXiv:math/0610777v3} (2007).

\bibitem{brouemallerouquier}
M.~Brou\'e, G.~Malle, and R.~Rouquier, \emph{Complex reflection groups, braid
  groups, {H}ecke algebras}, J.\ Reine Angew.\ Math. \textbf{500} (1998),
  127--190.

\bibitem{chevalley}
C.~Chevalley, \emph{Invariants of finite groups generated by reflections},
  Amer.\ J.\ Math. \textbf{77} (1955), no.~4, 778--782.

\bibitem{dunklopdam}
C.F. Dunkl and E.~Opdam, \emph{Dunkl operators for complex reflection groups},
  Proc.\ Lond.\ Math.\ Soc. \textbf{86} (2003), 70--108.

\bibitem{edelmanreiner}
P.H. Edelman and V.~Reiner, \emph{Free arrangements and rhombic tilings},
  Discrete Comput.\ Geom. \textbf{15} (1996), 307--340.

\bibitem{EG}
P.~Etingof and V.~Ginzburg, \emph{Symplectic reflection algebras,
  {C}alogero-{M}oser space and deformed {H}arish-{C}handra homomorphism},
  Invent.\ Math. \textbf{147} (2002).

\bibitem{fultonharris}
W.~Fulton and J.~Harris, \emph{Representation theory: a first course},
  Springer, New York (1991).

\bibitem{fuerlingerhofbauer}
J.~F{\"u}rlinger and J.~Hofbauer, \emph{$q$-{C}atalan numbers}, J.\ Combin.\
  Theory Ser.\ A \textbf{40} (1985), no.~2, 248--264.

\bibitem{garsiahaglund}
A.~Garsia and J.~Haglund, \emph{A positivity result in the theory of
  {M}acdonald polynomials}, Proc.\ Natl.\ Acad.\ Sci.\ USA \textbf{98} (2001),
  no.~8, 4313--4316.

\bibitem{garsiahaiman}
A.~Garsia and M.~Haiman, \emph{A remarkable $q,t$-{C}atalan sequence and
  $q$-{L}agrange inversion}, J.\ Algebraic Combin. \textbf{5} (1996), 191--244.

\bibitem{gordon}
I.~Gordon, \emph{On the quotient ring by diagonal harmonics}, Invent.\ Math.
  \textbf{153} (2003), 503--518.

\bibitem{gordonstafford}
I.~Gordon and J.T. Stafford, \emph{Rational {C}herednik algebras and {H}ilbert
  schemes}, Adv.\ Math. \textbf{198} (2005), no.~1, 222--274.

\bibitem{griffeth}
S.~Griffeth, \emph{Towards a combinatorial representation theory for the
  rational {C}herednik algebra of type {G}(r,p,n)}, to appear in Proc.\ Edinb.\
  Math.\ Soc., available at {\tt \scriptsize arXiv:math/0612733v3} (2008).

\bibitem{haglund2}
J.~Haglund, \emph{Conjectured statistics for the $q,t$-{C}atalan numbers},
  Adv.\ Math. \textbf{175} (2003), no.~2, 319--334.

\bibitem{haglund}
\bysame, \emph{The $q,t$-{C}atalan numbers and the space of diagonal
  harmonics}, University Lecture Series, Amer.\ Math.\ Soc. \textbf{41} (2008).

\bibitem{haglundloehr}
J.~Haglund and N.~Loehr, \emph{A conjectured combinatorial formula for the
  {H}ilbert series for diagonal harmonics}, Discrete Math. \textbf{298} (2005),
  189--204.

\bibitem{haiman2}
M.~Haiman, \emph{Conjectures on the quotient ring by diagonal invariants}, J.\
  Algebraic Combin. \textbf{3} (1994), 17--76.

\bibitem{haiman5}
\bysame, \emph{$t,q$-{C}atalan numbers and the {H}ilbert scheme}, Discrete
  Math., Selected papers in honor of Adriano Garsia \textbf{193} (1998),
  201--224.

\bibitem{haiman8}
\bysame, \emph{Combinatorics, symmetric functions, and {H}ilbert schemes},
  {CDM} Vol.~2002: current developments in mathematics (2002), 39--111.

\bibitem{haiman7}
\bysame, \emph{Notes on {M}acdonald polynomials and the geometry of {H}ilbert
  schemes}, {P}roceedings of the {NATO} {A}dvanced {S}tudy {I}nstitute,
  Cambridge (2002), 1--64.

\bibitem{haiman3}
\bysame, \emph{Vanishing theorems and character formulas for the {H}ilbert
  scheme of points in the plane}, Invent.\ Math. \textbf{149} (2002), 371--407.

\bibitem{haiman4}
\bysame, \emph{Commutative algebra of $n$ points in the plane}, Trends in
  Commutative Algebra, MSRI Publications \textbf{51} (2004), 153--180.

\bibitem{humphreys}
J.E. Humphreys, \emph{Reflection groups and {C}oxeter groups}, Cambridge Stud.\
  Adv.\ Math. \textbf{29} (1990).

\bibitem{loehr}
N.A. Loehr, \emph{Conjectured statistics for the higher $q,t$-{C}atalan
  sequences}, Electron.\ J.\ Combin. \textbf{12} (2005).

\bibitem{macaulay2}
\emph{Macaulay 2 -- {A} software system for research in algebraic geometry},
  available at {\tt \scriptsize http://www.math.uiuc.edu/Macaulay2/}.

\bibitem{macmahon2}
P.A. MacMahon, \emph{Collected papers: combinatorics {V}ol.\ 1}, MIT Press,
  Cambridge Mass. (1978).

\bibitem{ereiner}
E.~Reiner, \emph{Some applications of the theory of orbit harmonics}, PhD
  thesis, University of California, San Diego, USA (1993).

\bibitem{reiner}
V.~Reiner, \emph{Non-crossing partitions for classical reflection groups},
  Discrete Math. \textbf{177} (1997), 195--222.

\bibitem{shephard}
G.C. Shephard, \emph{Unitary groups generated by reflections}, Canad.\ J.\
  Math. \textbf{5} (1953), 364--383.

\bibitem{shephardtodd}
G.C. Shephard and J.A. Todd, \emph{Finite unitary reflection groups}, Canad.\
  J.\ Math. \textbf{6} (1954), 274--304.

\bibitem{shi}
J.-Y. Shi, \emph{The number of $\oplus$-sign types}, Quart.\ J.\ Math.\ Oxford
  \textbf{48} (1997), 375--390.

\bibitem{singular}
\emph{Singular -- {A} computer algebra system for polynomial computations},
  available at {\tt \scriptsize http://www.singular.uni-kl.de/impressum.html}.

\bibitem{springer}
T.A. Springer, \emph{Regular elements of finite reflection groups}, Invent.\
  Math. \textbf{25} (1974), 159--198.

\bibitem{stump3}
C.~Stump, \emph{$q,t$-{F}u\ss-{C}atalan numbers for complex reflection groups},
  in DMTCS as part of the FPSAC $2008$ conference proceedings, available at
  {\tt \scriptsize arXiv:0806.2936} (2008).

\bibitem{stump2}
\bysame, \emph{$q,t$-{F}u\ss-{C}atalan numbers for finite reflection groups},
  PhD thesis, University of Vienna, Vienna, Austria (2008).

\bibitem{yoshinaga}
M.~Yoshinaga, \emph{Characterization of a free arrangement and a conjecture of
  {E}delman and {R}einer}, Invent.\ Math. \textbf{157} (2004), 449--454.

\end{thebibliography}
\end{document}